\documentclass[a4paper,12pt,twoside]{article}

\title{On the conditional joint probability distributions of phase-type under the mixture of finite-state absorbing Markov jump processes}

 \author{B.A. Surya\footnote{School of Mathematics and Statistics, Victoria University of Wellington, Gate 6 Kelburn PDE, Wellington 6140, New Zealand. Email address: budhi.surya@vuw.ac.nz }\\ School of Mathematics and Statistics \\ Victoria University of Wellington, New Zealand }

\date{}

\usepackage{latexsym,epic,eepic,float}
\usepackage{color,colordvi}
\usepackage{shadow}
\usepackage{amsmath,amssymb}
\usepackage{amsmath} 
\usepackage{epsfig}
\usepackage{lscape}
\usepackage{rotating}
\usepackage{subfigure}
\usepackage{multicol}
\usepackage{tikz}
\usepackage{bbold}
\usetikzlibrary{automata,positioning}

\newtheorem{theorem}{Theorem}[section]
\newtheorem{theo}[theorem]{Theorem}
\newtheorem{lem}[theorem]{Lemma}

\newtheorem{cor}[theorem]{Corollary}
\newtheorem{prop}[theorem]{Proposition}
\newtheorem{Rem}[theorem]{Remark}
\newtheorem{Ex}[theorem]{Example}

\newcommand{\exit}{{\mbox{\, \vspace{3mm}}} \hfill\mbox{$\square$}}

\rm 

\numberwithin{equation}{section}

\date{12 May 2018}

\begin{document}

\maketitle \pagestyle{myheadings} \markboth{B.A. Surya}{On the conditional joint probability distributions of phase-type}

\begin{abstract}
This paper presents some new results on the conditional joint probability distributions of phase-type under the mixture of right-continuous Markov jump processes with absorption on the same finite state space $\mathbb{S}$ moving at different speeds, where the mixture occurs at a random time. Such mixture was first proposed by Frydman \cite{Frydman2005} and Frydman and Schuermann \cite{Frydman2008} as a generalization of the mover-stayer model of Blumen et at. \cite{Blumen}, and was recently extended by Surya \cite{Surya2018}. When conditioning on all previous and current information $\mathcal{F}_{t,i}=\mathcal{F}_{t-}\cup\{X_t=i\}$, with $\mathcal{F}_{t-}=\{X_s, 0<s\leq t-\}$ and $i\in\mathbb{S}$, of the mixture process $X$, distributional identities are explicit in terms of the Bayesian updates of switching probability, the likelihoods of observing the sample paths, and the intensity matrices of the underlying Markov processes, despite the fact that the mixture itself is non-Markov. They form non-stationary function of time and have the ability to capture heterogeneity and path dependence. When the underlying processes move at the same speed, in which case the mixture reduces to a simple Markov jump process, these features are removed, and the distributions coincide with that of given by Neuts \cite{Neuts1975} and Assaf et al. \cite{Assaf1984}. Furthermore, when conditioning on $\mathcal{F}_{t-}$ and no exit to the absorbing set has been observed at time $t$, the distributions are given explicitly in terms of an additional Bayesian updates of probability distribution of $X$ on $\mathbb{S}$. Examples are given to illustrate the main results. 
  
\medskip

\textbf{MSC2010 Subject Classification:} 60J20, 60J27, 60J28, 62N99


\end{abstract}

\section{Introduction}

Markov chain has been one among the most important probabilistic tools in modeling complex stochastic systems evolutions. It has been widely used in variety of applications across various fields such as, among others, in ecology (Balzter \cite{Balzter}), finance (Jarrow and Turnbull \cite{Jarrow1995}, and Jarrow et al. \cite{Jarrow1997}), marketing (Berger and Nasr \cite{Berger} and Pfeifer and Carraway \cite{Pfeifer}), etc. The phase-type model describes the lifetime distribution of an absorbing Markov chain. It was first introduced in univariate form by Neuts \cite{Neuts1975} in 1975 as the generalization of Erlang distribution. It has dense property, which can approximate any distribution of positive random variables arbitrarily well, and has closure property under finite convex mixtures and convolutions. When jumps distribution of compound Poisson process is modelled by phase-type model, it results in a dense class of L\'evy processes, see Asmussen \cite{Asmussen2003}. The advantage of working under phase-type distribution is that it allows some analytically tractable results in applications. To mention some, in option pricing (Asmussen et al. \cite{Asmussen2004}), actuarial science (Albrecher and Asmussen \cite{Asmussen2010}, Rolski et al. \cite{Rolski}, Zadeh et al. \cite{Zadeh}), in survival analysis (Aalen \cite{Aalen1995}, Aalen and Gjessing \cite{Aalen2001}), in queueing theory (Chakravarthy and Neuts \cite{Chakravarthy}, Asmussen \cite{Asmussen2003}), in reliability theory (Assaf and Levikson \cite{Assaf1982}, Okamura and Dohi \cite{Okamura}). 

The phase-type distribution $\overline{F}$ is expressed in terms of a Markov jump process $\{X_t\}_{t\geq 0}$ with finite state space $\mathbb{S}=E\cup \{\Delta\}$, where for some integer $m\geq 1$, $E=\{i: i=1,...,m\}$ and $\Delta$ represent respectively the transient and absorbing states. The lifetime of the Markov process and its distribution are defined by
\begin{equation}\label{eq:DefTime}
\tau=\inf\{t\geq 0: X_t=\Delta\} \quad \textrm{and} \quad \overline{F}(t)=\mathbb{P}\{\tau > t\}.
\end{equation}

In view of credit risk applications, the state space $\mathbb{S}$ represents the possible credit classes, with $1$ being the highest (\textrm{Aaa} in Moody's rankings) and $m$ being the lowest (\textrm{C} in Moody's rankings), whilst the absorbing state $\Delta$ represents bankruptcy, \textrm{D}. The distribution $\pi_k$ represents the proportion of homogeneous bonds in the rating $k$. We refer to \cite{Jarrow1995} and \cite{Jarrow1997} and literature therein for details.   

Unless stated otherwise, we assume for simplicity that the initial probability $\boldsymbol{\pi}$ of starting in any of the $m + 1$ phases has zero mass on the absorbing state $\Delta$, i.e., $\pi_{\Delta}=0$, so that $\mathbb{P}\{\tau>0\}=1$. We also refer to $\Delta$ as the $(m+1)$th element of the state space $\mathbb{S}$, i.e., $\Delta=m+1$. The speed at which the Markov process moves along the state space $\mathbb{S}$ is described by an intensity matrix $\mathbf{Q}$. This matrix has block partition according to the process moving in the transient state $E$ and in the absorbing state $\Delta$, which admits the following block-partitioned form:
\begin{equation}\label{eq:MatQ}
\mathbf{Q} = \left(\begin{array}{cc}
  \mathbf{A} & -\mathbf{A}\mathbb{1} \\
  \mathbf{0} & 0 \\
\end{array}\right),
\end{equation}
with $\mathbb{1}=(1,...,1)^{\top}$, as the rows of the intensity matrix $\mathbf{Q}$ sums to zero. That is to say that the entry $q_{ij}$ of the matrix $\mathbf{Q}$ satisfies the following properties:
\begin{equation}\label{eq:matq}
q_{ii}\leq 0, \; \; q_{ij}\geq 0, \; \; \sum_{j\neq i} q_{ij}=-q_{ii}=q_i, \quad (i,j)\in \mathbb{S}.
\end{equation}
As $-\mathbf{A}\mathbb{1}$ is a non-negative vector, (\ref{eq:matq}) implies that $\mathbf{A}$ to be a negative definite matrix, i.e., $\mathbb{1}^{\top}\mathbf{A}\mathbb{1}<0$. The matrix $\mathbf{A}$ is known as the phase generator matrix of $\mathbf{Q}$. The absorption is certain if and only if $\mathbf{A}$ is nonsingular, see Neuts \cite{Neuts1981}.

Following Theorem 3.4 and Corollary 3.5 in \cite{Asmussen2003} and by the homogeneity of $X$, the transition probability matrix $\mathbf{P}(t)$ of $X$ over the period of time $(0,t)$ is
\begin{equation}\label{eq:transprob}
\mathbf{P}(t)= \exp(\mathbf{Q} t), \quad t\geq 0.
\end{equation}

\noindent The entry $q_{ij}$ has probabilistic interpretation: $1/(-q_{ii})$ is the expected length of time that $X$ remains in state $i\in E$, and $q_{ij}/q_{i}$ is the probability that when a transition out of state $i$ occurs, it is to state $j\in\mathbb{S}$, $i\neq j$. The representation of the distribution $\overline{F}$ is uniquely specified by $(\boldsymbol{\pi},\mathbf{A})$. We refer among others to Neuts \cite{Neuts1981} and Asmussen \cite{Asmussen2003} for details. Following \cite{Neuts1981} and Proposition 4.1 \cite{Asmussen2003},
\begin{equation}\label{eq:DistrDefTime}
 \overline{F}(t)=\boldsymbol{\pi}^{\top} e^{\mathbf{A} t} \mathbb{1} \quad \textrm{and} \quad f(t)=-\boldsymbol{\pi}^{\top} e^{\mathbf{A} t} \mathbf{A}\mathbb{1}.
 \end{equation}

The extension of (\ref{eq:DistrDefTime}) to multivariate form was proposed by Assaf et al. \cite{Assaf1984} and later by Kulkarni \cite{Kulkarni}. Following \cite{Assaf1984}, let $\Gamma_1,...,\Gamma_n$ be nonempty stochastically closed subsets of $\mathbb{S}$ such that $\cap_{k=1}^n \Gamma_k$ is a proper subset of $\mathbb{S}$. ($\Gamma_i\subset \mathbb{S}$ is said to be stochastically closed if once $X$ enters $\Gamma_i$, it never leaves.) We assume without loss of generality that $\cap_{k=1}^n \Gamma_k$ consists of only the absorbing state $\Delta$, i.e., $\cap_{k=1}^n \Gamma_k=\Delta$. Since $\Gamma_k$ is stochastically closed, necessarily $q_{ij}=0$ if $i\in\Gamma_k$ and $j\in\Gamma_k^c$.

The first time until absorbtion of $X$ in the set $\Gamma_k$ is defined by
\begin{equation}\label{eq:MultiPH}
\tau_k:=\inf\{t\geq 0: X_t \in \Gamma_k\}.
\end{equation}
The joint distribution $\overline{F}$ of $\{\tau_k\}$ is called the multivariate phase type distribution, see \cite{Assaf1984}. Let $t_{i_n}\geq\dots\geq t_{i_1}\geq 0$ be the ordering of $(t_1,...,t_n)\in\mathbb{R}_+^n$. Following \cite{Assaf1984},
\begin{equation}\label{eq:MPH}
\begin{split}
\overline{F}(t_1,...,t_n)=&\mathbb{P}\{\tau_1 > t_1,..., \tau_n > t_n)=\boldsymbol{\pi}^{\top}\Big(\prod_{k=1}^n \exp\big(\mathbf{A}(t_{i_k}-t_{i_{k-1}})\big)\mathbf{H}_{k}\Big)\mathbb{1},
\end{split}
\end{equation}
where $\mathbf{H}_{k}$ is $(m\times m)$ diagonal matrix whose $i$th diagonal element, for $i=1,...,m$, equals $1$ when $i\in\Gamma_k^c$ and is zero otherwise. Again, as before we assume $\boldsymbol{\pi}$ has zero mass on $\Delta$ and $\pi_i\neq 0$ for $i\in\bigcup_{k=1}^n \Gamma_k^c$ implying that $\mathbb{P}\{\tau_1>0,..., \tau_n>0)=1$.

The multivariate distribution (\ref{eq:MPH}) has found various applications, e.g., in modeling credit default contagion (Herbertsson \cite{Herbertsson}, Bielecki et al. \cite{Bielecki}), in modeling aggregate loss distribution in insurance (Berdel and Hipp \cite{Berdel}, Asimit and Jones \cite{Asimit} and Willmot and Woo \cite{Willmot}), and in Queueing theory (Badila et al. \cite{Badila}).

Due to spatial homogeneity of the underlying Markov process, the distributions (\ref{eq:DistrDefTime}) \& (\ref{eq:MPH}) have stationary property and have therefore no ability to capture heterogeneity and available information of its past. In the recent empirical works of Frydman and Schuermann \cite{Frydman2008}, it was found that bonds of the same credit rating, represented by the state space of the Markov process, can move at different speeds to other ratings. Furthermore, the inclusion of past credit ratings improves out-of-sample prediction of the Nelson-Aalen estimate of credit default intensity. These empirical findings suggest that the dynamics of credit rating can be represented by a mixture $X$ of Markov jump processes moving with different speeds (intensity matrices), where the mixture itself is non-Markov. In his recent work, Surya \cite{Surya2018} extended the mixture model \cite{Frydman2008} and gave explicit distributional identities of the mixture process. However, the analyses performed in \cite{Frydman2008} and \cite{Surya2018} were based on knowing the initial and current states of the process. We extend the results by relaxing this assumption. For this purpose, we define by $\mathcal{G}_{t}=\mathcal{F}_{t-}\cup\{X_t\neq \Delta\}$ the set of all previous information and knowing that there is no exit to absorbing set $\{\Delta\}$ has been observed at time $t$, i.e., $\mathcal{G}_{t}=\bigcup_{i\in E}\mathcal{F}_{t,i}$. Conditional on $\mathcal{F}_{t,i}$ and $\mathcal{G}_t$, we derive explicit formula for the joint distributions 
\begin{equation}\label{eq:MPHnew}
\begin{split}
\overline{F}_{i,t}(t_1,...,t_n)=&\mathbb{P}\big\{\tau_1> t_1,..., \tau_n > t_n \big\vert \mathcal{F}_{t,i}\big\}\\
\overline{F}_t(t_1,...,t_n)=&\mathbb{P}\big\{\tau_1> t_1,..., \tau_n > t_n \big\vert \mathcal{G}_{t}\big\},
\end{split}
\end{equation}
under the mixture process $X$, with $1 \leq n\in\mathbb{N}$, $i\in\mathbb{S}$ and $0<t\leq \min\{t_n,...,t_1\}$. 

In view of credit risk applications \cite{Bielecki2002}, the quantity $\overline{F}_{i,t}(t_1,...,t_n)$ describes the joint distribution of exit times $\{\tau_k\}$ (\ref{eq:MultiPH}), due to cause-specific exits (default, prepayment, calling back, etc), of $i-$rated bonds, conditional on the credit rating history up to the current time $t$, whilst the function $\overline{F}_{t}(t_1,...,t_n)$ determines the joint distribution of the bonds' exit times across credit ratings viewed at the time $t$. In the framework of competing risks (see for instance Pintilie \cite{Pintilie}), for the observed exit time $\tau:=\min\{\tau_1,...,\tau_n\}$ and reason of exit $\boldsymbol{\xi}=\textrm{argmin}\{\tau_1,...,\tau_n\}$, the probability $\mathbb{P}\{t\leq \tau<s,\boldsymbol{\xi}=1\vert \mathcal{F}_{t,i}\}$ determines the proportion of $i-$rated bonds exiting by type $1$ from the pool within $s-t$ period of time, whilst $\mathbb{P}\{t\leq \tau<s,\boldsymbol{\xi}=1\vert \mathcal{G}_{t}\}$ represents the percentage of bonds exiting by type $1$.

The organization of this paper is as follows. We discuss in Section 2 the Markov mixture process in details. Section 3 presents some preliminaries, which extend further the results in \cite{Frydman2008} and \cite{Surya2018}. The main contributions of this paper are presented in Section 4. Some explicit examples of the results are discussed in Section 5, in which we show that the exit times $\{\tau_k\}$ are independent under the Markov model, but not for the mixture model. Section 6 concludes this paper.

\section{Mixture of Markov jump processes}
Throughout the remaining of this paper we denote by $X=\{X_t^{(\phi)}, t\geq0\}$ the mixture process, which is a continuous-time stochastic process defined as a mixture of two Markov jump processes $X^{(0)}=\{X_t^{(0)}, t\geq 0\}$ and $X^{(1)}=\{X_t^{(1)}, t\geq 0\}$, whose intensity matrices are given respectively by $\mathbf{Q}$ and $\mathbf{G}$. We assume that the underlying processes $X^{(0)}$ and $X^{(1)}$ are right-continuous. The two processes are defined on the same finite state space $\mathbb{S}$. It is defined following \cite{Surya2018} by
\begin{equation}\label{eq:mixture}
X=
\begin{cases}
X^{(0)}, & \phi=0\\
X^{(1)}, & \phi=1,
\end{cases}
\end{equation}
where the variable $\phi$ represents the speed regimes, assumed to be unobservable. 

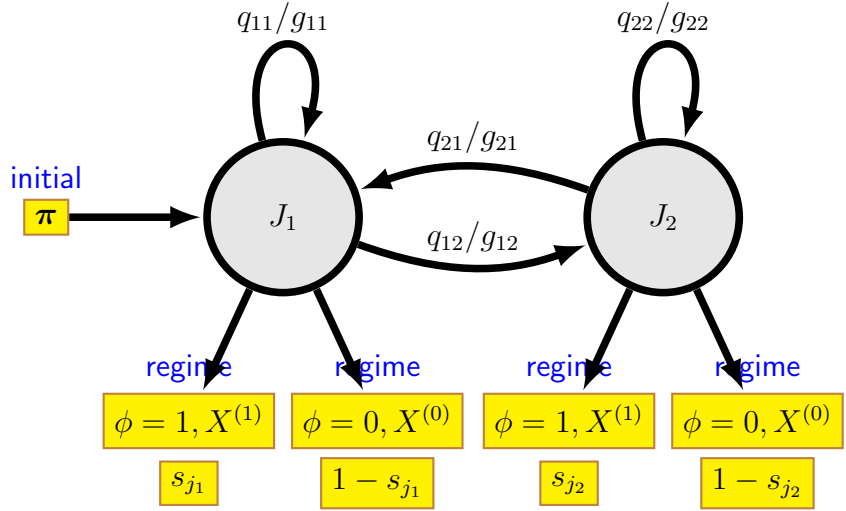
\begin{figure}
\begin{center}
  \begin{tikzpicture}[font=\sffamily]
 
        \tikzset{node style/.style={state, 
                                    minimum width=2cm,
                                    line width=1mm,
                                    fill=gray!20!white}}
                                                                                                 
          \tikzset{My Rectangle/.style={rectangle, draw=brown, fill=yellow, thick,
    prefix after command= {\pgfextra{\tikzset{every label/.style={blue}}, label=below}}
    }
}

          \tikzset{My Rectangle2/.style={rectangle,draw=brown,  fill=yellow, thick,
    prefix after command= {\pgfextra{\tikzset{every label/.style={blue}}, label=below}}
    }
}

          \tikzset{My Rectangle3/.style={rectangle, draw=brown, fill=yellow, thick,
    prefix after command= {\pgfextra{\tikzset{every label/.style={blue}}, label=below}}
    }
}
                   
        \node[node style] at (2, 0)     (s1)     {$J_1$};
        \node[node style] at (7, 0)     (s2)     {$J_2$};   
        
      \node [My Rectangle3, label={initial}] at  ([shift={(-5em,0em)}]s1.west) (p0) {$\boldsymbol{\pi}$};    
              
        \node [My Rectangle, label={regime} ] at  ([shift={(3em,-4em)}]s1.south) (g1) {$\phi=0, X^{(0)}$};
          \node [My Rectangle, label={regime} ] at  ([shift={(-3em,-4em)}]s1.south)  (g2) {$\phi=1, X^{(1)}$};
 
          \node [My Rectangle2] at  ([shift={(3em,-6em)}]s1.south) {$1-s_{j_1}$};
          \node [My Rectangle2] at  ([shift={(-3em,-6em)}]s1.south) {$s_{j_1}$};

         \node [My Rectangle, label={regime} ] at  ([shift={(3em,-4em)}]s2.south) (g3) {$\phi=0, X^{(0)}$};
          \node [My Rectangle, label={regime} ] at  ([shift={(-3em,-4em)}]s2.south)  (g4) {$\phi=1, X^{(1)}$};
          
           \node [My Rectangle2] at  ([shift={(3em,-6em)}]s2.south) {$1-s_{j_2}$};
          \node [My Rectangle2] at  ([shift={(-3em,-6em)}]s2.south) {$s_{j_2}$};
                   
        \draw[every loop,
              auto=right,
              line width=1mm,
              >=latex,
              draw=orange,
              fill=orange]
       
            (s1)  edge[bend right=20, auto=left] node {$q_{12}/g_{12}$} (s2)
            (s1)  edge[loop above]                     node {$q_{11}/g_{11}$} (s1)
            (s2)  edge[loop above]                     node {$q_{22}/g_{22}$} (s2)
            (s2)  edge[bend right=20]                node {$q_{21}/g_{21}$}  (s1)
            
            (s1) edge node {} (g1)
            (s1) edge node {} (g2)
            
            (s2) edge node {} (g3)
            (s2) edge node {} (g4)
              
            (p0) edge node {} (s1);
                                                
 \end{tikzpicture}
 \caption{State diagram of the Markov mixture process (\ref{eq:mixture}).}
\end{center}
\end{figure}


Markov mixture process is a generalization of mover-stayer model, a mixture of two discrete-time Markov chains proposed by Blumen et al \cite{Blumen} in 1955 to model population heterogeneity in jobs labor market. In the mover-stayer model \cite{Blumen}, the population of workers consists of stayers (workers who always stay in the same job category, $\mathbf{Q}=\mathbf{0}$) and movers (workers who move according to a stationary Markov chain with intensity matrix $\mathbf{G}$). Estimation of the mover-stayer model was discussed in Frydman \cite{Frydman1984}. Frydman \cite{Frydman2005} extended the model to a mixture of two continuous-time Markov chains moving with intensity matrices $\mathbf{Q}$ and $\mathbf{G}=\boldsymbol{\Psi}\mathbf{Q}$, where $\boldsymbol{\Psi}$ is a diagonal matrix. Frydman and Schuermann \cite{Frydman2008} later used the result to model the dynamics of firms' credit ratings. As empirically shown in \cite{Frydman2008}, there is strong evidence to suggest that firms of the same credit rating can move at different speeds to other credit ratings, a feature that lacks in the Markov model. Further distributional properties and identities of the mixture process were given in \cite{Surya2018}, in particular in the presence of absorbing states.

For a given initial state $i_0$, there is a separate mixing distribution defined by
\begin{equation}\label{eq:portion}
s_{i_0}=\mathbb{P}\{\phi=1 \vert X_0=i_0\} \quad \textrm{and} \quad 1-s_{i_0}=\mathbb{P}\{\phi=0 \vert X_0=i_0\},
\end{equation}
with $0\leq s_{i_0} \leq 1$. The quantity $s_{i_0}$ has the interpretation as the portion of firms with initial rating $i_0$ that evolve according to the process $X^{(1)}$, whilst $1-s_{i_0}$ is the proportion that propagates according to $X^{(0)}$. In general, $X^{(0)}$ and $X^{(1)}$ have different expected length of time the process occupies a state $i$, i.e., $1/q_{i}\neq 1/g_{i}$, but under \cite{Frydman2008} both processes have the same probability of leaving the state $i\in E$ to state $j\in\mathbb{S}$, $i\neq j$, i.e. $q_{ij}/q_{i}=g_{ij}/g_{i}$. Note that we have used $g_i$ and $g_{ij}$ to denote negative of the $i$th diagonal element and the $(i,j)$ entry of $\mathbf{G}$, respectively. Thus, depending on whether $0=\psi_i:=[\boldsymbol{\Psi}]_{i,i}$,  $0<\psi_i<1$, $\psi_i>1$ or $\psi_i=1$, $X^{(1)}$ never moves out of state $i$ (the mover-stayer model), moves out of state $i$ at lower rate, higher rate or at the same rate, subsequently, than that of $X^{(0)}$. If $\psi_i=1$, for all $i\in \mathbb{S}$, $X$ reduces to a simple Markov jump process $X^{(0)}$. 

Figure 1 illustrates the transition of $X$ (\ref{eq:mixture}) between states $J_1$ and $J_2$. When $X$ is observed in state $J_1$, it would stay in the state for an exponential period of time with intensity $q_{j_1}$ or $g_{j_1}$ before moving to state $J_2$ with probability $q_{j_1,j_2}/q_{j_1}$ or $g_{j_1,j_2}/g_{j_1}$ depending on whether it is driven by the Markov chain $X^{(0)}$ or $X^{(1)}$.

The main feature of the mixture process $X$ (\ref{eq:mixture}) is that unlike its component $X^{(0)}$ and $X^{(1)}$, $X$ does not have the Markov property; future development of its state depends on its past information. The section below discusses this further.

\section{Preliminaries}

Recall that the process $X$ (\ref{eq:mixture}) repeatedly changes its speed randomly in time either at the rate $\mathbf{Q}$ or $\mathbf{G}$. The speed regime, which is represented by the variable $\phi$, is however not directly observable; we can not classify from which regime the observed process $X$ came from. But, it can be identified based on past realizations of the process. We have denoted by $\mathcal{F}_{t-}$ all previous information about $X$ prior to time $t>0$, and by $\mathcal{F}_{t,i}=\mathcal{F}_{t-}\cup\{X_t=i\}$, $i\in \mathbb{S}$. The set $\mathcal{F}_{t-}$ may contain full observation, partial information or maybe nothing about the past of $X$.

The likelihoods of observing the past realization $\mathcal{F}_{t,j}$ of $X$ under $X^{(1)}$ and $X^{(0)}$ conditional on knowing its initial state $i$ are defined respectively by
\begin{equation}\label{eq:likelihood}
\begin{split}
L_{i,j}^Q(t):=&\mathbb{P}\{\mathcal{F}_{t,j} \vert \phi=0, X_0=i\}= \prod_{k\in \mathbb{S}} e^{-q_{k} T_k} \prod_{j\neq k, j\in \mathbb{S}} (q_{kj})^{N_{kj}},\\
L_{i,j}^G(t):=&\mathbb{P}\{\mathcal{F}_{t,j} \vert \phi=1, X_0=i\}= \prod_{k\in \mathbb{S}} e^{-g_{k} T_k} \prod_{j\neq k, j\in\mathbb{S}} (g_{kj})^{N_{kj}},
\end{split}
\end{equation}
where in the both expressions we have denoted subsequently by $T_k$ and $N_{kj}$ the total time the observed process $X$ spent in state $k\in \mathbb{S}$ for $\mathcal{F}_{t,j}$, and the number of transitions from state $k$ to state $j$, with $j\neq k$, observed in $\mathcal{F}_{t,j}$; whereas $q_{kj}$ and $g_{kj}$ represent the $(k,j)-$entry of the intensity matrices $\mathbf{Q}$ and $\mathbf{G}$, respectively.

\subsection{Bayesian updates of switching probability}

The Bayesian updates of switching probability $s_j(t)$ of $X$ (\ref{eq:mixture}) is defined by
\begin{equation}\label{eq:bayesianupdates}
s_j(t)=\mathbb{P}\{\phi=1\vert \mathcal{F}_{t,j}\}, \quad j\in\mathbb{S}, \; t\geq0.
\end{equation}
It represents the proportion of those in state $j$ moving according to Markov process $X^{(1)}$. Note that $s_j(0)=s_j$. Denote by $\widetilde{\mathbf{S}}(t)$ a diagonal matrix defined by 
\begin{equation}\label{eq:St}
\widetilde{\mathbf{S}}(t) = 
 \left(\begin{array}{cc}
 \mathbf{S}(t) & \mathbf{0} \\
  \mathbf{0} & s_{m+1}(t) \\
\end{array}\right),
\end{equation}
with $\mathbf{S}(t) =\mathrm{diag}(s_1(t), s_2(t),...,s_m(t))$, representing switching probability matrix of $X$.
For $t=0$, in which case $\mathcal{F}_{t,j}=\{X_0=j\}$, we write $\widetilde{\mathbf{S}}:= \widetilde{\mathbf{S}}(0)$, $\mathbf{S}:= \mathbf{S}(0)$. 

\pagebreak

Depending on the availability of the past information of $X$, the elements $s_j(t)$, $j\in \mathbb{S}$, of the switching probability matrix $\widetilde{\mathbf{S}}(t)$ (\ref{eq:St}) is given below.
\begin{prop}\label{prop:lem1}
Let $\boldsymbol{\pi}$ be the initial probability distribution of $X$ (\ref{eq:mixture}) on the state space $\mathbb{S}$. Define by $\mathbf{L}^{G}(t)$ and $\mathbf{L}^{Q}(t)$ the likelihood matrices whose $(i,j)$ elements $[\mathbf{L}^{G}(t)]_{i,j}$ and $[\mathbf{L}^{Q}(t)]_{i,j}$ are given by (\ref{eq:likelihood}). Then, for $j\in\mathbb{S}$ and $t\geq0$,
\begin{equation}\label{eq:likelihood2}
s_j(t)=\frac{\boldsymbol{\pi}^{\top}\widetilde{\mathbf{S}} \mathbf{L}^G(t)\mathbf{e}_j}{\boldsymbol{\pi}^{\top} \big[ \widetilde{\mathbf{S}} \mathbf{L}^G(t) + \big(\mathbf{I} -\widetilde{\mathbf{S}}\big) \mathbf{L}^Q(t) \big]\mathbf{e}_j}.
\end{equation}
To be more precise, depending on availability of information set $\mathcal{F}_{t-},$ we have
\begin{enumerate}
\item[(i)] Under full observation $\mathcal{F}_{t-}=\{X_s, 0<s\leq t-\}$ and conditional on knowing the initial state $i_0$, i.e., $\boldsymbol{\pi}=\mathbf{e}_{i_0}$, $s_j(t)$ simplifies further to 
\begin{equation*}
s_j(t)=\frac{s_{i_0}L_{i_0,j}^G(t)}{s_{i_0}L_{i_0,j}^G(t)+(1-s_{i_0})L_{i_0,j}^Q(t)}.
\end{equation*}
\item[(ii)] In case of $\mathcal{F}_{t-}=\emptyset$, $\mathbf{L}^Q(t)=e^{\mathbf{Q}t}$ and $\mathbf{L}^G(t)=e^{\mathbf{G}t}$, and therefore
\begin{equation*}
s_j(t)=\frac{\boldsymbol{\pi}^{\top} \widetilde{\mathbf{S}} e^{\mathbf{G}t}\mathbf{e}_j}{\boldsymbol{\pi}^{\top} \big[ \widetilde{\mathbf{S}} e^{\mathbf{G}t} + \big(\mathbf{I} - \widetilde{\mathbf{S}}\big) e^{\mathbf{Q}t} \big]\mathbf{e}_j}.
\end{equation*}
\item[(iii)] In case of $\mathcal{F}_{t-}=\emptyset$ and conditional on knowing the initial state $i_0$,
\begin{equation*}
s_j(t)=\frac{\mathbf{e}_{i_0}^{\top} \widetilde{\mathbf{S}} e^{\mathbf{G}t}\mathbf{e}_j}{\mathbf{e}_{i_0}^{\top} \big[ \widetilde{\mathbf{S}} e^{\mathbf{G}t} + \big(\mathbf{I} - \widetilde{\mathbf{S}}\big) e^{\mathbf{Q}t} \big]\mathbf{e}_j}.
\end{equation*}
\end{enumerate}
\end{prop}
The expression (\ref{eq:likelihood2}) for $s_j(t)$ generalizes the result of Lemma 3.1 in \cite{Surya2018}. It follows from (\ref{eq:likelihood2}) that when the underlying Markov processes $X^{(1)}$ and $X^{(0)}$ move at the same speed, i.e., $\mathbf{G}=\mathbf{Q}$, we see that $s_j(t)=1$ for all $j\in\mathbb{S}$ and $t\geq 0$, implying that the observed process $X$ is driven by the Markov chain $X^{(1)}$.   

\medskip

\begin{proof}[Proposition \ref{prop:lem1}]
By the law of total probability and the Bayes' formula,
\begin{align*}
\mathbb{P}\{\mathcal{F}_{t,j},\phi=1\}=&\sum_{i\in\mathbb{S}}\mathbb{P}\{X_0=i\}\mathbb{P}\{\phi=1\vert X_0=i\}\mathbb{P}\{\mathcal{F}_{t,j}\vert \phi=1,X_0=i\}\\
=&\sum_{i\in\mathbb{S}}\pi_i\times s_i \times L_{i,j}^G(t)
=\boldsymbol{\pi}^{\top} \widetilde{\mathbf{S}}\mathbf{L}^G(t)\mathbf{e}_j.
\end{align*}
Similarly, one will obtain after applying the same method as above that
\begin{equation*}
\mathbb{P}\{\mathcal{F}_{t,j},\phi=0\}=\boldsymbol{\pi}^{\top}\big(\mathbf{I}-\widetilde{\mathbf{S}}\big)\mathbf{L}^Q(t)\mathbf{e}_j.
\end{equation*}
The claim in (\ref{eq:likelihood2}) is finally established on account of the Bayes' formula:
\begin{align*}
s_j(t)=\mathbb{P}\{\phi=1\vert\mathcal{F}_{t,j}\}=\frac{\mathbb{P}\{\mathcal{F}_{t,j},\phi=1\}}{\mathbb{P}\{\mathcal{F}_{t,j},\phi=1\}+\mathbb{P}\{\mathcal{F}_{t,j},\phi=0\}}.  \exit
\end{align*}
\end{proof}

\subsection{Conditional transition probability matrix}
The following result on the transition probability matrix of $X$ plays an important role in deriving explicit identities for the joint probability distributions (\ref{eq:MPHnew}).
\begin{theo}[Surya \cite{Surya2018}]\label{theo:theo1}
For any $s\geq t\geq 0$, the conditional transition probability matrix $[\mathbf{P}(t,s)]_{i,j}:=\mathbb{P}\{X_s=j\vert\mathcal{F}_{t,i}\}$ of the mixture process $X$ is given by
\begin{equation}\label{eq:transM}
\mathbf{P}(t,s)= \widetilde{\mathbf{S}}(t)e^{\mathbf{G}(s-t)}
+ \big[\mathbf{I}-\widetilde{\mathbf{S}}(t)\big]e^{\mathbf{Q}(s-t)}.
\end{equation}
\end{theo}

It is clear from (\ref{eq:transM}) that $X$ does not inherit the Markov property of the underlying processes $X^{(1)}$ and $X^{(0)}$, i.e., future development of $X$ is determined by its past information $\mathcal{F}_{t,i}$ through its likelihoods. To be more precise, when we set $\mathbf{G}=\mathbf{Q}$ in (\ref{eq:transM}), $\mathbf{P}(t,s)=e^{\mathbf{Q}(s-t)}$, i.e., $X$ is a simple Markov jump process.

\section{Probability distributions of phase-type}\label{sec:mainsection}

This section presents the main results of this paper on the joint probability distributions of lifetime $\tau_k$ (\ref{eq:MultiPH}) under the mixture process $X$ (\ref{eq:mixture}), conditional on the information set $\mathcal{F}_{t,i}$ and $\mathcal{G}_{t}$. We assume that $X$ is defined on the finite state space $\mathbb{S}=E\cup\{\Delta\}$, where $E=\{1,2,...,m\}$ and $\Delta$ are transient and absorbing states, subsequently. We first discuss univariate $\mathcal{G}_{t}-$conditional distribution $\tau$ (\ref{eq:DefTime}) of $X$. To motivate and illustrate the main results on the multivariate distributions (\ref{eq:MPHnew}), we consider the bivariate case in some details. Throughout the remaining of this paper, define intensity matrices $\mathbf{G}$ and $\mathbf{Q}$ respectively by
\begin{equation}\label{eq:intensity}
\mathbf{G} = \left(\begin{array}{cc}
  \mathbf{B} & -\mathbf{B}\mathbb{1} \\
  \mathbf{0} & 0 \\
\end{array}\right)
\quad \textrm{and} \quad
\mathbf{Q} = \left(\begin{array}{cc}
 \mathbf{A} & -\mathbf{A}\mathbb{1} \\
  \mathbf{0} & 0 \\
\end{array}\right).
\end{equation}

The following results on block partition of the transition probability matrix $\mathbf{P}(t,s)$ (\ref{eq:transM}) and exponential matrix will be used to derive the conditional distributions (\ref{eq:MPHnew}) in closed form. We refer to Proposition 3.7 in \cite{Surya2018} for details. 

\begin{lem}
Let the phase generator matrix $\mathbf{A}$ be nonsingular. Then,
\begin{align}\label{eq:blockpartisi}
e^{\mathbf{Q}}=
\left(\begin{array}{cc}
e^{\mathbf{A}} & \mathbb{1}-e^{\mathbf{A}}\mathbb{1}\\
  \mathbf{0} & 1 \\
\end{array}\right).
\end{align}
\end{lem}

\begin{prop}
The transition probability matrix (\ref{eq:transM}) has block partition:
\begin{equation}\label{eq:blockPts}
\mathbf{P}(t,s)= \left(\begin{array}{cc}
\mathbf{P}_{11}(t,s) &\mathbf{P}_{12}(t,s)\\
  \mathbf{0} & 1 \\
\end{array}\right),
\end{equation}
where the matrix entries $\mathbf{P}_{11}(t,s)$ and $\mathbf{P}_{12}(t,s)$ are respectively defined by
\begin{align*}
\mathbf{P}_{11}(t,s)\;=\;&\mathbf{S}(t)e^{\mathbf{B}(s-t)}+ \big[\mathbf{I}-\mathbf{S}(t)\big]e^{\mathbf{A}(s-t)}\\
\mathbf{P}_{12}(t,s)\;=\;&\mathbf{S}(t)\big(\mathbf{I}-e^{\mathbf{B}(s-t)}\big)\mathbb{1} + \big[\mathbf{I}-\mathbf{S}(t)\big]\big(\mathbf{I}-e^{\mathbf{A}(s-t)}\big)\mathbb{1}.
\end{align*}
\end{prop}

Furthermore, in the sequel below we denote by $\boldsymbol{\pi}(t)$ the time$-t$ probability distribution of $X$ on the state space $\mathbb{S}$, whose $i$th element $\pi_i(t)$ is defined by
\begin{equation}
\pi_i(t)=\mathbb{P}\{X_t=i \big\vert \mathcal{G}_t\}, \; \textrm{for $i\in E$}, \quad \textrm{and} \quad \pi_i(t)=0, \; \textrm{for $i\in \Delta$}.
\end{equation}

\subsection{Bayesian updates of probability distribution $\boldsymbol{\pi}$}
The following proposition gives the distribution $\boldsymbol{\pi}(t)$ of $X$ on $E$ at time $t\geq0$.  
\begin{prop}\label{prop:propinitialprobab}
For a given $j\in E$ and $t\geq 0$, define $\pi_j(t)=\mathbb{P}\{X_t=j\vert \mathcal{G}_{t}\}$. 
\begin{align}\label{eq:piatt}
\pi_j(t)=\frac{\boldsymbol{\pi}^{\top} \big[\mathbf{S} \mathbf{L}^G(t) + \big(\mathbf{I} -\mathbf{S}\big) \mathbf{L}^Q(t) \big]\mathbf{e}_j}{\boldsymbol{\pi}^{\top} \big[\mathbf{S} \mathbf{L}^G(t) + \big(\mathbf{I} -\mathbf{S}\big) \mathbf{L}^Q(t) \big]\mathbb{1}}.
\end{align}
To be more precise, depending on availability of information set $\mathcal{F}_{t-},$ we have
\begin{enumerate}
\item[(i)] Under full observation $\mathcal{F}_{t-}=\{X_s, 0<s\leq t-\}$ and conditional on knowing the initial state $i_0$, i.e., $\boldsymbol{\pi}=\mathbf{e}_{i_0}$, $\pi_j(t)$ simplifies further to 
\begin{equation*}
\pi_j(t)=\frac{s_{i_0}L_{i_0,j}^G(t) + (1-s_{i_0})L_{i_0,j}^Q(t)}{\sum_{j\in E}\big(s_{i_0}L_{i_0,j}^G(t) + (1-s_{i_0})L_{i_0,j}^Q(t) \big)}.
\end{equation*}
\item[(ii)] In case of $\mathcal{F}_{t-}=\emptyset$, $\mathbf{L}^Q(t)=e^{\mathbf{Q}t}$ and $\mathbf{L}^G(t)=e^{\mathbf{G}t}$, and therefore
\begin{equation*}
\pi_j(t)=\frac{\boldsymbol{\pi}^{\top}\big[ \mathbf{S} e^{\mathbf{B}t} + (\mathbf{I}-\mathbf{S}) e^{\mathbf{A}t}\big] \mathbf{e}_j}{\boldsymbol{\pi}^{\top} \big[\mathbf{S} e^{\mathbf{B}t} + (\mathbf{I} -\mathbf{S}) e^{\mathbf{A}t} \big]\mathbb{1}}.
\end{equation*}
\item[(iii)] In case of $\mathcal{F}_{t-}=\emptyset$ and conditional on knowing the initial state $i_0$,
\begin{equation*}
\pi_j(t)=\frac{\mathbf{e}_{i_0}^{\top}\big[ \mathbf{S} e^{\mathbf{B}t} + (\mathbf{I}-\mathbf{S}) e^{\mathbf{A}t}\big]\mathbf{e}_j}{\mathbf{e}_{i_0}^{\top} \big[\mathbf{S} e^{\mathbf{B}t} + (\mathbf{I} -\mathbf{S}) e^{\mathbf{A}t} \big]\mathbb{1}}.
\end{equation*}
\end{enumerate}
\end{prop}

It follows that $0<\pi_j(t)<1$, $\sum_{j\in E}\pi_j(t)=1$ for $t\geq 0$, and $\boldsymbol{\pi}=\boldsymbol{\pi}(0)$. 

\medskip

\begin{proof}
The proof follows from applying the law of total probability and the Bayes' formula for conditional probability. By applying the latter, we have
\begin{align*}
\mathbb{P}\{\mathcal{F}_{t,j},\phi=1,X_0=i\}=&\mathbb{P}\{X_0=i\}\mathbb{P}\{\phi=1\vert X_0=i\}\mathbb{P}\{\mathcal{F}_{t,j}\vert \phi=1,X_0=i\}\\
=&\pi_i\times s_i\times L_{i,j}^G(t). 
\end{align*}
By the same approach, $\mathbb{P}\{\mathcal{F}_{t,j},\phi=0,X_0=i\}=\pi_i\times (1-s_i)\times L_{i,j}^Q(t).$ Hence, %
\begin{align*}
\mathbb{P}\{\mathcal{F}_{t,j},X_0=i\}=&\mathbb{P}\{\mathcal{F}_{t,j},\phi=1,X_0=i\}+\mathbb{P}\{\mathcal{F}_{t,j},\phi=0,X_0=i\}\\
=&\pi_i\times s_i\times L_{i,j}^G(t)+\pi_i\times (1-s_i)\times L_{i,j}^Q(t).
\end{align*}
Therefore, we have by the above and applying the law of total probability that 
\begin{align*}
\mathbb{P}\{\mathcal{F}_{t,j}\}=&\sum_{i\in \mathbb{S}}\mathbb{P}\{\mathcal{F}_{t,j},X_0=i\}\\
=&\boldsymbol{\pi}^{\top}\big(\mathbf{S}\mathbf{L}^G(t) + \big[\mathbf{I}-\mathbf{S}\big]\mathbf{L}^Q(t)\big)\mathbf{e}_j.
\end{align*}
The result (\ref{eq:piatt}) is established by the Bayes' rule and the law of total probability,
\begin{align*}
\pi_j(t)=\mathbb{P}\{X_t=j\vert \mathcal{G}_{t}\}=\frac{\mathbb{P}\{\mathcal{F}_{t,j}\}}{\sum_{k\in E}\mathbb{P}\{\mathcal{F}_{t,k}\}},
\end{align*}
whereas the cases $(ii)$ and $(iii)$ follow on account of (\ref{eq:blockpartisi}) and that $\pi_{\Delta}=0$. \exit
\end{proof}

\medskip

The result of Proposition \ref{prop:propinitialprobab} gives an additional feature to the distributional properties of the mixture of Markov jump processes discussed in \cite{Surya2018} and \cite{Frydman2008}.

\subsection{Univariate conditional phase-type distributions}

In this section we derive an explicit identity for $\overline{F}_t(s)=\mathbb{P}\{\tau > s \vert \mathcal{G}_{t}\}$, for $s\geq t\geq 0$, which extends the conditional probability (see Theorem 4.1 in \cite{Surya2018}):
\begin{align}\label{eq:PH0}
\overline{F}_{i,t}(s)=\mathbb{P}\{\tau > s \vert \mathcal{F}_{t,i}\}=\mathbf{e}_i^{\top}\Big(\mathbf{S}(t) e^{\mathbf{B}(s-t)}+ \big[\mathbf{I}-\mathbf{S}(t)\big] e^{\mathbf{A}(s-t)}\Big)\mathbb{1}.
\end{align}
\begin{lem}
The $\mathcal{G}_{t}-$conditional distribution $\overline{F}_t(s)$ is given for $s\geq t\geq 0$ by
\begin{equation}\label{eq:PH1}
\overline{F}_t(s) = \boldsymbol{\pi}^{\top}(t)\Big(\mathbf{S}(t) e^{\mathbf{B}(s-t)}+ \big[\mathbf{I}-\mathbf{S}(t)\big] e^{\mathbf{A}(s-t)}\Big)\mathbb{1}.
\end{equation}
\end{lem}
\begin{proof}
As $\tau$ is the time until absorption of $X$, by the law of total probability,
\begin{align*}
\mathbb{P}\{\tau>s \big\vert \mathcal{G}_t\}=&\sum_{j,k\in E}\Big(\mathbb{P}\big\{X_s=j, X_t= k, \phi=1 \big\vert X_t\neq \Delta, \mathcal{F}_{t-}\big\}\\
&\hspace{1cm}+\mathbb{P}\big\{X_s=j, X_t= k, \phi = 0\big\vert X_t\neq \Delta, \mathcal{F}_{t-}\big\}\Big).
\end{align*}
On account that $\{X_t=k\}\subset \{X_t\neq \Delta\}$, for $k\in E$, by the Bayes' formula,
\begin{align*}
&\mathbb{P}\big\{X_s=j, X_t= k, \phi=1 \big\vert X_t\neq \Delta, \mathcal{F}_{t-}\big\}\\
&\hspace{2cm}=\mathbb{P}\big\{X_t=k\big\vert X_t\neq \Delta, \mathcal{F}_{t-}\big\}\mathbb{P}\big\{\phi=1\big\vert X_t=k,\mathcal{F}_{t-}\big\}\\
&\hspace{3cm}\mathbb{P}\big\{X_s=j \big\vert \phi=1, X_t=k,\mathcal{F}_{t-}\big\}\\
&\hspace{2cm}=\pi_k(t) s_k(t) \mathbf{e}_k^{\top} e^{\mathbf{G}(s-t)} \mathbf{e}_j^{\top}.
\end{align*}
Applying similar arguments for the above derivation, one can obtain
\begin{align*}
\mathbb{P}\big\{X_s=j, X_t= k, \phi=0 \big\vert X_t\neq \Delta, \mathcal{F}_{t-}\big\}=&\pi_k(t) \big(1-s_k(t)\big) \mathbf{e}_k^{\top} e^{\mathbf{Q}(s-t)} \mathbf{e}_j^{\top}.
\end{align*}
The claim in (\ref{eq:PH1}) is established by applying the transition matrix (\ref{eq:blockPts}). \exit
\end{proof}
\begin{Rem}
Following the two identities (\ref{eq:PH1}) and (\ref{eq:PH0}), we can conclude that
\begin{align}\label{eq:relation}
\mathbb{P}\{\tau > s \vert \mathcal{G}_{t}\}=\sum_{i\in E} \pi_i(t)\mathbb{P}\{\tau > s \vert \mathcal{F}_{t,i}\}.
\end{align}
\end{Rem}

\medskip

Following the same approach discussed in \cite{Neuts1975}, \cite{Neuts1981} and \cite{Surya2018}, the density function $f_t(s)$ of $\tau$, its Laplace transform and $n$th moment are presented below.

\begin{theo}
The $\mathcal{G}_{t}-$conditional density $f_t(s)$ of $\tau$ is given for $s\geq t\geq 0$ by
\begin{equation}\label{eq:PHD1}
f_t(s) = -\boldsymbol{\pi}^{\top}(t)\Big(\mathbf{S}(t) e^{\mathbf{B}(s-t)}\mathbf{B}+ \big[\mathbf{I}-\mathbf{S}(t)\big] e^{\mathbf{A}(s-t)} \mathbf{A}\Big)\mathbb{1}.
\end{equation}
\begin{enumerate}
\item[(i)] The Laplace transform $\widehat{F}_t[\lambda]=\int_0^{\infty} e^{-\lambda u} f_t(t+u) du$ is given by
\begin{equation*}
\widehat{F}_t[\lambda]=-\boldsymbol{\pi}^{\top}(t)\Big(\mathbf{S}(t)\big(\lambda\mathbf{I}-\mathbf{B}\big)^{-1}\mathbf{B}
+\big[\mathbf{I}-\mathbf{S}(t)\big]\big(\lambda\mathbf{I}-\mathbf{A}\big)^{-1}\mathbf{A}\Big)\mathbb{1}.
\end{equation*}
\item[(ii)] The $\mathcal{G}_{t}-$conditional $n$th moment, \textrm{for $n=0,1,...$}, of $\tau$ is given by
\begin{equation*}
\mathbb{E}\{\tau^n \vert \mathcal{G}_{t}\}=(-1)^n n!\boldsymbol{\pi}^{\top}(t)\Big(\mathbf{S}(t) \mathbf{B}^{-n}+\big[\mathbf{I}-\mathbf{S}(t)\big]\mathbf{A}^{-n}\Big)\mathbb{1}.
\end{equation*}
\end{enumerate}
\end{theo}

\pagebreak 

\noindent Setting $\mathbf{B}=\mathbf{A}$, in which case the mixture process is driven by $X^{(0)}$, the above results coincide with that of given in \cite{Neuts1975} and Proposition 4.1 in \cite{Asmussen2003} for $t=0$. 

\medskip

The following results summarize the dense and closure properties under finite convex mixtures and finite convolutions of $F_t(s)$ (\ref{eq:PH1}). They can be established using matrix analytic approach \cite{Neuts1981}. See for e.g. Theorems 4.12 and 4.13 in \cite{Surya2018}.  

\begin{theo}
The phase-type distribution $F_t(s)$ (\ref{eq:PH1}) is closed under finite convex mixtures and convolutions, and forms a dense class of distributions on $\mathbb{R}_+$.
\end{theo}

\subsection{Bivariate conditional phase-type distributions}

As in the univariate case, we consider the mixture process $X$ (\ref{eq:mixture}) on the finite state space $\mathbb{S}=E\cup\{\Delta\}$. Following \cite{Assaf1984}, let $\boldsymbol{\Gamma}_1$ and $\boldsymbol{\Gamma}_2$ be two nonempty stochastically closed subsets of $\mathbb{S}$ such that $\boldsymbol{\Gamma}_1\cap \boldsymbol{\Gamma}_2$ is proper subset of $\mathbb{S}$. We assume without loss of generality that $\boldsymbol{\Gamma}_1\cap \boldsymbol{\Gamma}_2=\Delta$ and the absorption into $\Delta$ is certain, i.e., the generator matrices $\mathbf{A}$ and $\mathbf{B}$ need to be nonsingular. As $\boldsymbol{\Gamma}_k$, $k=1,2$, is stochastically closed set, necessarily $[\mathbf{Q}]_{i,j}=0=[\mathbf{G}]_{i,j}$ if $i\in\boldsymbol{\Gamma}_k$ and $j\in \boldsymbol{\Gamma}_k^c$.

We denote by $\boldsymbol{\pi}$ the initial probability vector on $\mathbb{S}$ such that $\pi_{\Delta}=0$. We shall assume that $\boldsymbol{\pi}_i\neq 0$ if $i\in \boldsymbol{\Gamma}_1^c\cup \boldsymbol{\Gamma}_2^c$ implying $\mathbb{P}\{\tau_1>0, \tau_2>0\}=1$. As before, $\mathcal{F}_{t,i}=\mathcal{F}_{t-}\cup\{X_t=i\}$ defines all previous and current information of $X$. 

\subsubsection{The conditional joint survival function of $\tau_1$ and $\tau_2$}

The joint distribution of $\tau_k$ (\ref{eq:MPHnew}), for $k=1,2$, are given by the following.
\begin{lem}\label{lem:lemjointCDF}
The identity for $\mathcal{F}_{t,i}-$conditional joint distribution $\overline{F}_{i,t}(t_1,t_2)=\mathbb{P}\{\tau_1>t_1, \tau_2>t_2 \vert \mathcal{F}_{t,i}\}$ of $\tau_1$ and $\tau_2$ is given for $t_1,t_2\geq t\geq 0$ and $i\in E$ by
\begin{align*}
\overline{F}_{i,t}(t_1,t_2)=
\begin{cases}
\overline{F}_{i,t}^{(1)}(t_1,t_2):=\mathbf{e}_i^{\top}\Big(\mathbf{S}(t)e^{\mathbf{B}(t_{2}-t)}\mathbf{H}_2e^{\mathbf{B}(t_1-t_2)}\mathbf{H}_1 \\
&\hspace{-7cm}+\big[\mathbf{I}-\mathbf{S}(t)\big]e^{\mathbf{A}(t_{2}-t)}\mathbf{H}_2e^{\mathbf{A}(t_1-t_2)}\mathbf{H}_1\Big)\mathbb{1}, \; \textrm{if $t_1\geq t_2\geq t \geq 0$} \\[8pt]
\overline{F}_{i,t}^{(2)}(t_1,t_2):=\mathbf{e}_i^{\top}\Big(\mathbf{S}(t)e^{\mathbf{B}(t_{1}-t)}\mathbf{H}_1e^{\mathbf{B}(t_2-t_1)}\mathbf{H}_2 \\
&\hspace{-7cm}+\big[\mathbf{I}-\mathbf{S}(t)\big]e^{\mathbf{A}(t_1-t)}\mathbf{H}_1e^{\mathbf{A}(t_2-t_1)}\mathbf{H}_2\Big)\mathbb{1}, \; \textrm{if $t_2\geq t_1\geq t \geq 0$}.
\end{cases}
\end{align*}
Note that we have used the notation $\mathbf{H}_k$ to denote a $(m\times m)-$diagonal matrix whose $i$th diagonal element for $i=1,2,...,m$ equals $1$ if $i\in\Gamma_k^c$ and is $0$ otherwise.
\end{lem}
\begin{proof}
To begin with let $(t_{i_1},t_{i_2})$, with $t_{i_2}\geq t_{i_1}$ be the ordering of $(t_1,t_2)$, with $t_{i_1}\geq t_{i_0}=t$. Since $\tau_{i_k}$, $k=1,2$, is the time until absorption of $X$ (\ref{eq:mixture}) into $\Gamma_{i_k}$, 
\begin{align}
\mathbb{P}\{\tau_1>t_1,\tau_2>t_2 \big\vert \mathcal{F}_{t_{i_0},i}\}=&\mathbb{P}\{\tau_{i_1}>t_{i_1}, \tau_{i_2}>t_{i_2} \big\vert \mathcal{F}_{t_{i_0},i}\} \nonumber\\
=& \mathbb{P}\{X_{t_{i_1}}\in\boldsymbol{\Gamma}_{i_1}^c, X_{t_{i_2}}\in\boldsymbol{\Gamma}_{i_2}^c \big\vert \mathcal{F}_{t_{i_0},i} \nonumber\}\\
=&\sum_{J_{i_1}\in\Gamma_{i_1}^c}\sum_{J_{i_2}\in\Gamma_{i_2}^c} \mathbb{P}\{X_{t_{i_1}}=J_{i_1}, X_{t_{i_2}}=J_{i_2} \big\vert  \mathcal{F}_{t_{i_0},i}\}. \label{eq:turunan1}
\end{align}
The probability on the r.h.s of the last equality can be worked out as follows.
\begin{align*}
&\mathbb{P}\big\{X_{t_{i_1}}=J_{i_1}, X_{t_{i_2}}=J_{i_2} \big\vert  \mathcal{F}_{t_{i_0},i}\big\}\\
&\hspace{2cm}= \mathbb{P}\big\{X_{t_{i_0}}=J_{i_0}\vert \mathcal{F}_{t_{i_0},i}\big\}\mathbb{P}\big\{\phi=1 \big\vert X_{t_{i_0}}=J_{i_0},\mathcal{F}_{t_{i_0},i}\big\} \\
&\hspace{3.5cm}\times \mathbb{P}\big\{X_{t_{i_1}}=J_{i_1} \big\vert \phi=1, X_{t_{i_0}}=J_{i_0},\mathcal{F}_{t_{i_0},i}\big\}\\
&\hspace{4.5cm}\times \mathbb{P}\big\{X_{t_{i_2}}=J_{i_2} \big\vert \phi=1, X_{t_{i_1}}=J_{i_1}, X_{t_{i_0}}=J_{i_0},\mathcal{F}_{t_{i_0},i}\big\}\\
&\hspace{2cm}+ \mathbb{P}\big\{X_{t_{i_0}}=J_{i_0}\vert \mathcal{F}_{t_{i_0},i}\big\}\mathbb{P}\big\{\phi=0 \big\vert X_{t_{i_0}}=J_{i_0},\mathcal{F}_{t_{i_0},i}\big\} \\
&\hspace{3.5cm}\times \mathbb{P}\big\{X_{t_{i_1}}=J_{i_1} \big\vert \phi=0, X_{t_{i_0}}=J_{i_0},\mathcal{F}_{t_{i_0},i}\big\}\\
&\hspace{4.5cm}\times \mathbb{P}\big\{X_{t_{i_2}}=J_{i_2} \big\vert \phi=0, X_{t_{i_1}}=J_{i_1}, X_{t_{i_0}}=J_{i_0},\mathcal{F}_{t_{i_0},i}\big\}.
\end{align*}
Note that we have applied the law of total probability and Bayes' rule for conditional probability in the above equality. Recall that $\mathbb{P}\big\{X_{t_{i_0}}=J_{i_0}\vert \mathcal{F}_{t_{i_0},i}\big\}=1$ iff $J_{i_0}=i$ and zero otherwise. In terms of the Bayesian updates (\ref{eq:bayesianupdates}) we have:
\begin{align*}
&\mathbb{P}\big\{X_{t_{i_1}}=J_{i_1}, X_{t_{i_2}}=J_{i_2} \big\vert  \mathcal{F}_{t_{i_0},i}\big\}\\
&\hspace{2cm}=\mathbf{e}_i^{\top}\mathbf{S}(t) e^{\mathbf{G}(t_{i_1}-t_{i_0})}\mathbf{e}_{J_{i_1}} \mathbf{e}_{J_{i_1}}^{\top} e^{\mathbf{G}(t_{i_2}-t_{i_1})} \mathbf{e}_{J_{i_2}}  \mathbf{e}_{J_{i_2}}^{\top}\mathbb{1}\\
&\hspace{3cm}+\mathbf{e}_i^{\top}\big[\mathbf{I}-\mathbf{S}(t)\big] e^{\mathbf{Q}(t_{i_1}-t_{i_0})}\mathbf{e}_{J_{i_1}} \mathbf{e}_{J_{i_1}}^{\top} e^{\mathbf{Q}(t_{i_2}-t_{i_1})} \mathbf{e}_{J_{i_2}}  \mathbf{e}_{J_{i_2}}^{\top}\mathbb{1}.
\end{align*}
Therefore, starting from equation (\ref{eq:turunan1}), we have following the above that
\begin{align*}
&\mathbb{P}\{\tau_{i_1}>t_{i_1}, \tau_{i_2}>t_{i_2} \big\vert \mathcal{F}_{t_{i_0},i}\}=\sum_{J_{i_1}\in\Gamma_{i_1}^c}\sum_{J_{i_2}\in\Gamma_{i_2}^c} \mathbb{P}\{X_{t_{i_1}}=J_{i_1}, X_{t_{i_2}}=J_{i_2} \big\vert  \mathcal{F}_{t_{i_0},i}\}\\
&\hspace{1cm}=\mathbf{e}_i^{\top}\mathbf{S}(t) e^{\mathbf{G}(t_{i_1}-t_{i_0})}\Big(\sum_{J_{i_1}\in\Gamma_{i_1}^c} \mathbf{e}_{J_{i_1}}\mathbf{e}_{J_{i_1}}^{\top} \Big) e^{\mathbf{G}(t_{i_2}-t_{i_1})} \Big(\sum_{J_{i_2}\in\Gamma_{i_2}^c}  \mathbf{e}_{J_{i_2}}  \mathbf{e}_{J_{i_2}}^{\top}\Big)\mathbb{1}  \\
&\hspace{1.5cm}+ \mathbf{e}_i^{\top} \big[\mathbf{I}-\mathbf{S}(t) \big] e^{\mathbf{Q}(t_{i_1}-t_{i_0})}\Big(\sum_{J_{i_1}\in\Gamma_{i_1}^c} \mathbf{e}_{J_{i_1}}\mathbf{e}_{J_{i_1}}^{\top} \Big) e^{\mathbf{Q}(t_{i_2}-t_{i_1})} \Big(\sum_{J_{i_2}\in\Gamma_{i_2}^c}  \mathbf{e}_{J_{i_2}}  \mathbf{e}_{J_{i_2}}^{\top}\Big)\mathbb{1},
\end{align*}
leading to the form $\overline{F}_{i,t}(t_1,t_2)$ on account of $\mathbf{H}_{i_k}=\sum\limits_{J_{i_k}\in\Gamma_{i_k}^c} \mathbf{e}_{J_{i_k}}\mathbf{e}_{J_{i_k}}^{\top}$ and (\ref{eq:blockpartisi}). \exit
\end{proof}
\begin{prop}
The distribution $\overline{F}_{t}(t_1,t_2)=\mathbb{P}\{\tau_1>t_1, \tau_2>t_2 \vert \mathcal{G}_{t}\}$ is given by
\begin{align*}
\overline{F}_{t}(t_1,t_2)=
\begin{cases}
\overline{F}_{t}^{(1)}(t_1,t_2):=\boldsymbol{\pi}^{\top}(t)\Big(\mathbf{S}(t)e^{\mathbf{B}(t_{2}-t)}\mathbf{H}_2e^{\mathbf{B}(t_1-t_2)}\mathbf{H}_1 \\
&\hspace{-7.25cm}+\big[\mathbf{I}-\mathbf{S}(t)\big]e^{\mathbf{A}(t_{2}-t)}\mathbf{H}_2e^{\mathbf{A}(t_1-t_2)}\mathbf{H}_1\Big)\mathbb{1}, \; \textrm{if $t_1\geq t_2\geq t\geq 0$} \\[8pt]
\overline{F}_{t}^{(2)}(t_1,t_2):=\boldsymbol{\pi}^{\top}(t)\Big(\mathbf{S}(t)e^{\mathbf{B}(t_{1}-t)}\mathbf{H}_1e^{\mathbf{B}(t_2-t_1)}\mathbf{H}_2 \\
&\hspace{-7.25cm}+\big[\mathbf{I}-\mathbf{S}(t)\big]e^{\mathbf{A}(t_1-t)}\mathbf{H}_1e^{\mathbf{A}(t_2-t_1)}\mathbf{H}_2\Big)\mathbb{1},  \; \textrm{if $t_2\geq t_1\geq t \geq 0$}.
\end{cases}
\end{align*}
\end{prop}
\begin{proof}
By (\ref{eq:relation}) and law of total probability, $F_t(t_1,t_2)=\sum\limits_{i\in E} \pi_i(t)F_{i,t}(t_1,t_2)$. \exit
\end{proof}

\subsubsection{The conditional joint probability density function}

In general, the joint distribution $\overline{F}_{i,t}(t_1,t_2)$ (resp. $\overline{F}_{t}(t_1,t_2)$) has a singular component $\overline{F}_{i,t}^{(0)}(t_1,t_2)$ (resp. $\overline{F}_{t}^{(0)}(t_1,t_2)$) on the set $\{(t_1,t_2): t_2=t_1\}$. The singular component can be obtained by deriving the joint density of $\tau_1$ and $\tau_2$ and deduce the absolutely continuous and singular parts of the pdf, such as discussed in the theorem below. For non-matrix based bivariate function, see for instance \cite{Sarhan}.
\begin{theo}\label{theo:maintheo}
Given the joint distribution $\overline{F}_{i,t}(t_1,t_2)$ of $(\tau_1,\tau_2)$ as specified in Lemma \ref{lem:lemjointCDF}, the joint probability density $f_{i,t}(t_1,t_2)$ of $(\tau_1,\tau_2)$ is given by
\begin{eqnarray} \label{eq:jointpdfit}
f_{i,t}(t_1,t_2)=\frac{\partial^2 \overline{F}_{i,t}(t_1,t_2)}{\partial t_2\partial t_1}=
\begin{cases}
f_{i,t}^{(1)}(t_1,t_2), &\; \textrm{if $t_1\geq t_2\geq t \geq 0$} \\[8pt]
f_{i,t}^{(2)}(t_1,t_2), &\; \textrm{if $t_2\geq t_1\geq t \geq 0$}\\[8pt]
f_{i,t}^{(0)}(t_1,t_1), &\; \textrm{if $t_1= t_2\geq t \geq 0$}, \\
\end{cases}
\end{eqnarray}
where the absolutely continuous components $f_{i,t}^{(1)}(t_1,t_2)$ and $f_{i,t}^{(2)}(t_1,t_2)$ are
\begin{align*}
f_{i,t}^{(1)}(t_1,t_2)=&\mathbf{e}_i^{\top}\Big(\mathbf{S}(t)e^{\mathbf{B}(t_{2}-t)}\big[\mathbf{B},\mathbf{H}_2\big]e^{\mathbf{B}(t_1-t_2)}\mathbf{B}\mathbf{H}_1 \\
&+\big[\mathbf{I}-\mathbf{S}(t)\big]e^{\mathbf{A}(t_{2}-t)}\big[\mathbf{A},\mathbf{H}_2\big]e^{\mathbf{A}(t_1-t_2)}\mathbf{A}\mathbf{H}_1\Big)\mathbb{1},\\[8pt]
f_{i,t}^{(2)}(t_1,t_2)=&\mathbf{e}_i^{\top}\Big(\mathbf{S}(t)e^{\mathbf{B}(t_1-t)}\big[\mathbf{B},\mathbf{H}_1\big]e^{\mathbf{B}(t_2-t_1)}\mathbf{B}\mathbf{H}_2 \\
&+\big[\mathbf{I}-\mathbf{S}(t)\big]e^{\mathbf{A}(t_1-t)}\big[\mathbf{A},\mathbf{H}_1\big]e^{\mathbf{A}(t_2-t_1)}\mathbf{A}\mathbf{H}_2\Big)\mathbb{1},
\end{align*}
where the matrix operator $[A,B]=AB-BA$ defines the commutator of $A$ and $B$, whilst the singular component part $f_{i,t}^{(0)}(t_1,t_2)$ is defined by the function
\begin{align*}
f_{i,t}^{(0)}(t_1,t_1)=&\mathbf{e}_i^{\top}\Big\{\mathbf{S}(t)e^{\mathbf{B}(t_1-t)}\Big(\big[\mathbf{B},\mathbf{H}_2\big]\mathbf{H}_1 + \big[\mathbf{B},\mathbf{H}_1\big]\mathbf{H}_2 -\mathbf{B}\Big)\\
&+\big[\mathbf{I}-\mathbf{S}(t)\big]e^{\mathbf{A}(t_1-t)}\Big(\big[\mathbf{A},\mathbf{H}_2\big]\mathbf{H}_1 + \big[\mathbf{A},\mathbf{H}_1\big]\mathbf{H}_2 -\mathbf{A}\Big)\Big\}\mathbb{1}.
\end{align*}
\end{theo}
\begin{proof}
The expression for $f_{i,t}^{(1)}(t_1,t_2)$ (resp. $f_{i,t}^{(2)}(t_1,t_2)$) follows from applying the partial derivative $\frac{\partial^2}{\partial t_2\partial t_1}$ to $\overline{F}_{i,t}^{(1)}(t_1,t_2)$ (resp. to $\overline{F}_{i,t}^{(2)}(t_1,t_2)$) taking account that
\begin{align}\label{eq:dervexpm}
\frac{d}{dt} \Big( \mathbf{S}e^{\mathbf{B}t}\mathbf{A} \Big)=\mathbf{S}\mathbf{B}e^{\mathbf{B}t}\mathbf{A}=\mathbf{S}e^{\mathbf{B}t}\mathbf{B}\mathbf{A}.
\end{align}
To get $f_{i,t}^{(0)}(t_1,t_2)$, recall that $ \int_0^{\infty}e^{\mathbf{B}t} dt = -\mathbf{B}^{-1}$ due to $\mathbf{B}$ negative definite, and
\begin{align*}
1=&\int_t^{\infty} \int_t^{t_1} f_{i,t}^{(1)}(t_1,t_2) dt_2 dt_1 +
\int_t^{\infty} \int_t^{t_2} f_{i,t}^{(2)}(t_1,t_2) dt_1 dt_2 \\&+ \int_t^{\infty} f_{i,t}^{(0)}(t_1,t_1) dt_1.
\end{align*}
Applying Fubini's theorem, the first integral is given after some calculations by
\begin{align*}
\int_t^{\infty} \int_t^{t_1} f_{i,t}^{(1)}(t_1,t_2) dt_2 dt_1 =& \int_t^{\infty} \int_{t_2}^{\infty} f_{i,t}^{(1)}(t_1,t_2) dt_1 dt_2 \\
&\hspace{-4cm}=\int_t^{\infty}\int_{t_2}^{\infty} \Big(\mathbf{e}_i^{\top}\mathbf{S}(t)e^{\mathbf{B}(t_2-t)}\big[\mathbf{B},\mathbf{H}_2\big]e^{\mathbf{B}(t_1-t_2)}\mathbf{B}\mathbf{H}_1\mathbb{1} \Big)dt_1dt_2\\
&\hspace{-2.5cm}+\int_t^{\infty}\int_{t_2}^{\infty} \Big(\mathbf{e}_i^{\top}\big[\mathbf{I}-\mathbf{S}(t)\big]e^{\mathbf{A}(t_2-t)}\big[\mathbf{A},\mathbf{H}_2\big]e^{\mathbf{A}(t_1-t_2)}\mathbf{A}\mathbf{H}_1\mathbb{1} \Big)dt_1dt_2\\
&\hspace{-4cm}=\mathbf{e}_i^{\top}\Big(\mathbf{S}(t)\mathbf{B}^{-1}\big[\mathbf{B},\mathbf{H}_2\big] + \big[\mathbf{I}-\mathbf{S}(t)\big]\mathbf{A}^{-1}\big[\mathbf{A},\mathbf{H}_2\big]\Big)\mathbf{H}_1\mathbb{1}\\
&\hspace{-4cm}= -\int_t^{\infty}\mathbf{e}_i^{\top}\Big(\mathbf{S}(t)e^{\mathbf{B}(t_1-t)}\big[\mathbf{B},\mathbf{H}_2\big] + \big[\mathbf{I}-\mathbf{S}(t)\big]e^{\mathbf{A}(t_1-t)}\big[\mathbf{A},\mathbf{H}_2\big]\Big)\mathbf{H}_1\mathbb{1}dt_1.
\end{align*}
Following the same approach, one can show after some calculations that
\begin{align*}
\int_t^{\infty} \int_t^{t_2} f_{i,t}^{(2)}(t_1,t_2) dt_1 dt_2=&\mathbf{e}_i^{\top}\Big(\mathbf{S}(t)\mathbf{B}^{-1}\big[\mathbf{B},\mathbf{H}_1\big] + \big[\mathbf{I}-\mathbf{S}(t)\big] \mathbf{A}^{-1}\big[\mathbf{A},\mathbf{H}_1\big]\Big)\mathbf{H}_2\mathbb{1}\\
&\hspace{-4cm}=-\int_t^{\infty}\mathbf{e}_i^{\top}\Big(\mathbf{S}(t)e^{\mathbf{B}(t_1-t)}\big[\mathbf{B},\mathbf{H}_1\big] + \big[\mathbf{I}-\mathbf{S}(t)\big] e^{\mathbf{A}(t_1-t)}\big[\mathbf{A},\mathbf{H}_1\big]\Big)\mathbf{H}_2\mathbb{1} dt_1.
\end{align*}
The proof is established on account of the following fact that
\begin{align*}
-\int_t^{\infty}\mathbf{e}_i^{\top}\Big(\mathbf{S}(t)e^{\mathbf{B}(t_1-t)}\mathbf{B} + \big[\mathbf{I}-\mathbf{S}(t)\big]e^{\mathbf{A}(t_1-t)}\mathbf{A}\Big)\mathbb{1}dt_1=1. \quad \exit
\end{align*}
\end{proof}
\begin{prop}
For $t\geq 0$, the $\mathcal{G}_{t}-$conditional density $f_t(t_1,t_2)$ is given by
\begin{align}\label{eq:jointpdft}
f_{t}(t_1,t_2)=\frac{\partial^2 \overline{F}_{t}(t_1,t_2)}{\partial t_2\partial t_1}=
\begin{cases}
f_{t}^{(1)}(t_1,t_2), &\; \textrm{if $t_1\geq t_2\geq t \geq 0$} \\[8pt]
f_{t}^{(2)}(t_1,t_2), &\; \textrm{if $t_2\geq t_1\geq t \geq 0$} \\[8pt]
f_{t}^{(0)}(t_1,t_2), &\; \textrm{if $t_1= t_2\geq t \geq 0$}, \\
\end{cases}
\end{align}
where the absolutely continuous components $f_{t}^{(1)}(t_1,t_2)$ and $f_{t}^{(2)}(t_1,t_2)$ are
\begin{align*}
f_{t}^{(1)}(t_1,t_2)=&\boldsymbol{\pi}^{\top}(t)\Big(\mathbf{S}(t)e^{\mathbf{B}(t_{2}-t)}\big[\mathbf{B},\mathbf{H}_2\big]e^{\mathbf{B}(t_1-t_2)}\mathbf{B}\mathbf{H}_1 \\
&+\big[\mathbf{I}-\mathbf{S}(t)\big]e^{\mathbf{A}(t_{2}-t)}\big[\mathbf{A},\mathbf{H}_2\big]e^{\mathbf{A}(t_1-t_2)}\mathbf{A}\mathbf{H}_1\Big)\mathbb{1},\\[8pt]
f_{t}^{(2)}(t_1,t_2)=&\boldsymbol{\pi}^{\top}(t)\Big(\mathbf{S}(t)e^{\mathbf{B}(t_1-t)}\big[\mathbf{B},\mathbf{H}_1\big]e^{\mathbf{B}(t_2-t_1)}\mathbf{B}\mathbf{H}_2 \\
&+\big[\mathbf{I}-\mathbf{S}(t)\big]e^{\mathbf{A}(t_1-t)}\big[\mathbf{A},\mathbf{H}_1\big]e^{\mathbf{A}(t_2-t_1)}\mathbf{A}\mathbf{H}_2\Big)\mathbb{1},
\end{align*}
whilst the singular component $f_{t}^{(0)}(t_1,t_2)$ is defined by the function
\begin{align*}
f_{t}^{(0)}(t_1,t_2)=&\boldsymbol{\pi}^{\top}(t)\Big\{\mathbf{S}(t)e^{\mathbf{B}(t_1-t)}\Big(\big[\mathbf{B},\mathbf{H}_2\big]\mathbf{H}_1 + \big[\mathbf{B},\mathbf{H}_1\big]\mathbf{H}_2 -\mathbf{B}\Big)\\
&+\big[\mathbf{I}-\mathbf{S}(t)\big]e^{\mathbf{A}(t_1-t)}\Big(\big[\mathbf{A},\mathbf{H}_2\big]\mathbf{H}_1 + \big[\mathbf{A},\mathbf{H}_1\big]\mathbf{H}_2 -\mathbf{A}\Big)\Big\}\mathbb{1}.
\end{align*}
\end{prop}
\begin{proof}
It follows from identity (\ref{eq:relation}) that $f_t(t_1,t_2)=\sum_{i=1}^m \pi_i(t) f_{i,t}(t_1,t_2).$ \exit
\end{proof}

\pagebreak

\begin{Rem}
By setting $\mathbf{B}=\mathbf{A}$ in (\ref{eq:jointpdfit}) and (\ref{eq:jointpdft}) and taking the limit $t\rightarrow 0$ in the latter, we arrive at the bivariate distribution given in Assaf et al. \cite{Assaf1984}.
\end{Rem}
\begin{cor}\label{cor:jointCDFit}
The singular component of $\overline{F}_{i,t}(t_1,t_2)$ and $\overline{F}_{t}(t_1,t_2)$ are 
\begin{align*}
\overline{F}_{i,t}^{(0)}(t_1,t_1)=&
\mathbf{e}_i^{\top}\Big\{\mathbf{S}(t)e^{\mathbf{B}(t_1-t)}\mathbf{B}^{-1}\Big(\mathbf{B} - \big[\mathbf{B},\mathbf{H}_2\big]\mathbf{H}_1 - \big[\mathbf{B},\mathbf{H}_1\big]\mathbf{H}_2 \Big)\\
&\hspace{-1.5cm}+\big[\mathbf{I}-\mathbf{S}(t)\big]e^{\mathbf{A}(t_1-t)}\mathbf{A}^{-1}\Big(\mathbf{A} -\big[\mathbf{A},\mathbf{H}_2\big]\mathbf{H}_1 - \big[\mathbf{A},\mathbf{H}_1\big]\mathbf{H}_2\Big)\Big\}\mathbb{1}\\[8pt]
\overline{F}_{t}^{(0)}(t_1,t_1)=&
\boldsymbol{\pi}^{\top}(t)\Big\{\mathbf{S}(t)e^{\mathbf{B}(t_1-t)}\mathbf{B}^{-1}\Big(\mathbf{B} - \big[\mathbf{B},\mathbf{H}_2\big]\mathbf{H}_1 - \big[\mathbf{B},\mathbf{H}_1\big]\mathbf{H}_2 \Big)\\
&\hspace{-1.5cm}+\big[\mathbf{I}-\mathbf{S}(t)\big]e^{\mathbf{A}(t_1-t)}\mathbf{A}^{-1}\Big(\mathbf{A} -\big[\mathbf{A},\mathbf{H}_2\big]\mathbf{H}_1 - \big[\mathbf{A},\mathbf{H}_1\big]\mathbf{H}_2\Big)\Big\}\mathbb{1}.
\end{align*}
\end{cor}

Hence, the singular component of $\overline{F}_{it}(t_1,t_2)$ and $\overline{F}_{t}(t_1,t_2)$ is zero if and only if $[\mathbf{A}]_{i,j}=0=[\mathbf{B}]_{i,j}$ for $i\in \boldsymbol{\Gamma}_1^c \cap \boldsymbol{\Gamma}_2^c$ and $j=\Delta$, which is equivalent to imposing:
\begin{align}\label{eq:singularcond}
\mathbf{A} -\big[\mathbf{A},\mathbf{H}_2\big]\mathbf{H}_1 - \big[\mathbf{A},\mathbf{H}_1\big]\mathbf{H}_2=0= \mathbf{B} -\big[\mathbf{B},\mathbf{H}_2\big]\mathbf{H}_1 - \big[\mathbf{B},\mathbf{H}_1\big]\mathbf{H}_2.
\end{align}

\subsubsection{The conditional joint Laplace transform of $\tau_1$ and $\tau_2$}

In order to compute the $\mathcal{F}_{t,i}-$conditional moment $\mathbb{E}\big\{\tau_1^n\tau_2^m\big\vert \mathcal{F}_{t,i}\big\}$, it is therefore convenient to study the $\mathcal{F}_{t,i}-$conditional joint Laplace transform of $\tau_1$ and $\tau_2$:
\begin{align}\label{eq:jointMGF}
\Psi_{i,t}(\lambda_1,\lambda_2):=\mathbb{E}\big\{e^{-\lambda_1 \tau_1 - \lambda_2 \tau_2}\big\vert \mathcal{F}_{t,i}\big\} \quad \textrm{for $i\in E$}.
\end{align}

\begin{theo}
The $\mathcal{F}_{t,i}-$conditional joint Laplace transform $\Psi_{i,t}(\lambda_1,\lambda_2)$ of the exit times $\tau_1$ and $\tau_2$ is given for $\lambda_1,\lambda_2\geq 0$, $t\geq 0$ and $i\in E$ by 
\begin{align*}
\Psi_{i,t}(\lambda_1,\lambda_2)=&\mathbf{e}_i^{\top}\Big\{\mathbf{S}(t)\big((\lambda_1+\lambda_2)\mathbf{I}-\mathbf{B}\big)^{-1}\Big( [\mathbf{B},\mathbf{H}_2]\big(\lambda_1\mathbf{I}-\mathbf{B}\big)^{-1}\mathbf{B}\mathbf{H}_1 \\
&\hspace{0cm}+ [\mathbf{B},\mathbf{H}_1]\big(\lambda_2\mathbf{I}-\mathbf{B}\big)^{-1}\mathbf{B}\mathbf{H}_2  + [\mathbf{B},\mathbf{H}_2]\mathbf{H}_1 + [\mathbf{B},\mathbf{H}_1]\mathbf{H}_2-\mathbf{B}     \Big) \\
&+ \big[\mathbf{I}- \mathbf{S}(t)\big]\big((\lambda_1+\lambda_2)\mathbf{I}-\mathbf{A}\big)^{-1}\Big( [\mathbf{A},\mathbf{H}_2]\big(\lambda_1\mathbf{I}-\mathbf{A}\big)^{-1}\mathbf{A}\mathbf{H}_1 \\
&\hspace{0cm}+ [\mathbf{A},\mathbf{H}_1]\big(\lambda_2\mathbf{I}-\mathbf{A}\big)^{-1}\mathbf{A}\mathbf{H}_2  + [\mathbf{A},\mathbf{H}_2]\mathbf{H}_1 + [\mathbf{A},\mathbf{H}_1]\mathbf{H}_2-\mathbf{A}     \Big) \Big\}\mathbb{1}.
\end{align*}
\end{theo}
\begin{proof}
To start with, recall that for $i\in E$, $f_{i,t}(t_1,t_2)=0$ for $t_1,t_2<t$. Therefore,  
\begin{align*}
\Psi_{i,t}(\lambda_1,\lambda_2)=&
\int_0^{\infty}\int_0^{\infty} e^{-\lambda_1 u_1} e^{-\lambda_2 u_2} f_{it}(t+u_1,t+u_2)du_1du_2 \\
&= \int_0^{\infty}\int_0^{u_1} e^{-\lambda_1 u_1} e^{-\lambda_2 u_2} f_{it}^{(1)}(t+u_1,t+u_2)du_2du_1 \\
&\hspace{2cm}+ \int_0^{\infty}\int_0^{u_2} e^{-\lambda_1 u_1} e^{-\lambda_2 u_2} f_{it}^{(2)}(t+u_1,t+u_2)du_1du_2 \\
&\hspace{3.5cm}+ \int_0^{\infty} e^{-(\lambda_1 + \lambda_2)u_1} f_{it}^{(0)}(t+u_1,t+u_1)du_1.
\end{align*}
The proof is established after applying Fubini's theorem to each integral. \exit
\end{proof}

\pagebreak

By the law of total probability and Bayes' rule we have the following result.
\begin{prop}
The $\mathcal{G}_{t}-$conditional joint Laplace transform $\Psi_{t}(\lambda_1,\lambda_2):=\mathbb{E}\big\{e^{-\lambda_1 \tau_1 - \lambda_2 \tau_2}\big\vert \mathcal{G}_{t}\big\}$ of the exit times $\tau_1$ and $\tau_2$ is given for $\lambda_1,\lambda_2, t\geq 0$ by 
\begin{align*}
\Psi_{t}(\lambda_1,\lambda_2)=&\boldsymbol{\pi}^{\top}(t) \Big\{\mathbf{S}(t)\big((\lambda_1+\lambda_2)\mathbf{I}-\mathbf{B}\big)^{-1}\Big( [\mathbf{B},\mathbf{H}_2]\big(\lambda_1\mathbf{I}-\mathbf{B}\big)^{-1}\mathbf{B}\mathbf{H}_1 \\
&\hspace{0cm}+ [\mathbf{B},\mathbf{H}_1]\big(\lambda_2\mathbf{I}-\mathbf{B}\big)^{-1}\mathbf{B}\mathbf{H}_2  + [\mathbf{B},\mathbf{H}_2]\mathbf{H}_1 + [\mathbf{B},\mathbf{H}_1]\mathbf{H}_2-\mathbf{B}     \Big) \\
&+ \big[\mathbf{I}- \mathbf{S}(t)\big]\big((\lambda_1+\lambda_2)\mathbf{I}-\mathbf{A}\big)^{-1}\Big( [\mathbf{A},\mathbf{H}_2]\big(\lambda_1\mathbf{I}-\mathbf{A}\big)^{-1}\mathbf{A}\mathbf{H}_1 \\
&\hspace{0cm}+ [\mathbf{A},\mathbf{H}_1]\big(\lambda_2\mathbf{I}-\mathbf{A}\big)^{-1}\mathbf{A}\mathbf{H}_2  + [\mathbf{A},\mathbf{H}_2]\mathbf{H}_1 + [\mathbf{A},\mathbf{H}_1]\mathbf{H}_2-\mathbf{A}     \Big) \Big\}\mathbb{1}.
\end{align*}
\end{prop}

\medskip

\noindent Following the joint Laplace transform (\ref{eq:jointMGF}), we obtain the joint moments:
\begin{align*}
\mathbb{E}\big\{\tau_1^n \tau_2^m \big\vert \mathcal{F}_{t,i}\big\}=& (-1)^{m+n} \frac{\partial^{m+n}}{\partial \lambda_1^m \partial \lambda_2^n}\Psi_{i,t}(\lambda_1,\lambda_2)\Big\vert_{\lambda_1=0,\lambda_2=0}.
\end{align*}
\begin{Ex}
The conditional joint moments $\mathbb{E}\{\tau_1\tau_2\vert \mathcal{G}_{t}\}$ is given by
\begin{align*}
\mathbb{E}\{\tau_1\tau_2\vert \mathcal{G}_{t}\}=&2 ! \boldsymbol{\pi}^{\top}(t)\Big\{\mathbf{S}(t)\mathbf{B}^{-2} + \big[\mathbf{I}-\mathbf{S}(t)\big]\mathbf{A}^{-2}\Big\}\mathbb{1}\\
&\hspace{0cm}+\boldsymbol{\pi}^{\top}(t)\Big\{ \mathbf{S}(t)\mathbf{B}^{-2}\Big(\big[\mathbf{B},\mathbf{H}_1\big]\mathbf{B}^{-1}\mathbf{H}_2 + \big[\mathbf{B},\mathbf{H}_2\big]\mathbf{B}^{-1}\mathbf{H}_1\Big)\\
&\hspace{0.75cm}+
\big[\mathbf{I}-\mathbf{S}(t)\big]\mathbf{A}^{-2}\Big(\big[\mathbf{A},\mathbf{H}_1\big]\mathbf{A}^{-1}\mathbf{H}_2 + \big[\mathbf{A},\mathbf{H}_2\big]\mathbf{A}^{-1}\mathbf{H}_1\Big)
\Big\} \mathbb{1}.
\end{align*}
\end{Ex}

\subsection{Multivariate conditional phase-type distributions}

The extension to multivariate case follows similar approach to the bivariate one. Let $\Gamma_1,...,\Gamma_n$ be nonempty stochastically closed subsets of $\mathbb{S}$ such that $\cap_{k=1}^n \Gamma_k$ is a proper subset of $\mathbb{S}$. Without loss of generality, we assume that $\cap_{k=1}^n \Gamma_k=\Delta$. Since $\Gamma_k$ is stochastically closed, we necessarily assume that $q_{ij}=0$ (and therefore $g_{ij}=0$) if $i\in\Gamma_k$ and $j\in\Gamma_k^c$, for $k\in\{1,...,n\}$, and $\boldsymbol{\pi}_i\neq 0$ whenever $i\in \cup_{k=1}^n \boldsymbol{\Gamma}_k^c$.

Furthermore, denote by $\tau_k$ the first entry time of $X$ in the set $\boldsymbol{\Gamma}_k$ defined in (\ref{eq:MultiPH}). To formulate the joint distribution of $\tau_k$, let $(t_{i_1},...,t_{i_n})$ be the time ordering of $(t_1,...,t_n)\in\mathbb{R}_+^n$, where $(i_1,...,i_n)$ is a permutation of $(1,2,...,n)$. Subsequently, we define by $j_{i_k}\in \boldsymbol{\Gamma}_{i_k}^c$ the state that $X$ occupies at time $t=t_{i_k}$. 

\begin{lem}\label{lem:main}
Let $t_{i_n}\geq \dots \geq t_{i_1}\geq t_{i_0}=t\geq 0$ be the time ordering of $(t_1,...,t_n)\in\mathbb{R}_+^n$. The joint distribution of the exit times $\tau_k$ (\ref{eq:MultiPH}) is given by
\begin{equation}\label{eq:maincdf1}
\begin{split}
&\hspace{2.65cm}\overline{F}_{i,t}(t_{i_1},...,t_{i_n})=\mathbb{P}\big\{\tau_{i_1}>t_{i_1},...,\tau_{i_n}>t_{i_n} \big\vert \mathcal{F}_{t,i}\big\}\\
&\hspace{-0.15cm}=\mathbf{e}_i^{\top}\Big\{\mathbf{S}(t)\prod_{k=1}^n e^{\mathbf{B}(t_{i_k}-t_{i_{k-1}})}\mathbf{H}_{i_k} + \big[\mathbf{I}-\mathbf{S}(t)\big]  \prod_{k=1}^ne^{\mathbf{A}(t_{i_k}-t_{i_{k-1}})} \mathbf{H}_{i_k}\Big\}\mathbb{1},
\end{split}
\end{equation}
where $\mathbf{H}_k$ is an $(m\times m)-$ diagonal matrix whose $i$th element $[\mathbf{H}_k]_{i,i}=\mathbb{1}_{\{i\in\boldsymbol{\Gamma}_k^c\}}.$
\end{lem}
\begin{proof}
Following similar arguments of the proof in bivariate case, we obtain
\begin{align}
\mathbb{P}\big\{\tau_1>t_1,...,\tau_n>t_n \big \vert \mathcal{F}_{t,i}\big\}=&\mathbb{P}\big\{\tau_{i_1}>t_{i_1},...,\tau_{i_n}>t_{i_n} \big\vert \mathcal{F}_{t_{i_0},i}\big\} \nonumber \\
&\hspace{-2cm}=\mathbb{P}\big\{X_{t_{i_0}}=j_{i_0},X_{t_{i_1}}\in\boldsymbol{\Gamma}_{i_1}^c,...,X_{t_{i_n}}\in\boldsymbol{\Gamma}_{i_n}^c  \big\vert \mathcal{F}_{t_{i_0},i} \big\} \label{eq:Derivation}\\
&\hspace{-2cm}=\sum_{j_{i_1}\in\boldsymbol{\Gamma}_{i_1}^c}...\sum_{j_{i_n}\in\boldsymbol{\Gamma}_{i_n}^c} \mathbb{P}\big\{X_{t_{i_0}}=j_{i_0},X_{t_{i_1}}=j_{i_1},..., X_{t_{i_n}}=j_{i_n}  \big\vert \mathcal{F}_{t_{i_0},i} \big\}.  \nonumber
\end{align}
By Bayes' theorem for conditional probability and the law of total probability,
\begin{align*}
&\mathbb{P}\big\{X_{t_{i_0}}=j_{i_0},X_{t_{i_1}}=j_{i_1},..., X_{t_{i_n}}=j_{i_n}  \big\vert \mathcal{F}_{t_{i_0},i} \big\}\\
&\hspace{1.5cm}=\mathbb{P}\big\{X_{t_{i_0}}=J_{i_0} \big \vert \mathcal{F}_{t_{i_0},i}\big\}
\times \mathbb{P}\big\{\phi=1 \big\vert X_{t_{i_0}}=J_{i_0},\mathcal{F}_{t_{i_0},i}\big\}\\
&\hspace{3cm}\times \mathbb{P}\big\{X_{t_{i_1}}=J_{i_1} \big\vert \phi=1, X_{t_{i_0}}=J_{i_0},\mathcal{F}_{t_{i_0},i} \big\} \\
&\hspace{3.2cm}\vdots \\
&\hspace{3cm} 
\times \mathbb{P}\big\{X_{t_{i_n}}=J_{i_n} \big\vert \phi=1, X_{t_{i_{n-1}}}=J_{i_{n-1}},\dots,X_{t_{i_0}}=J_{i_0},\mathcal{F}_{t_{i_0},i} \big\} \\
&\hspace{1.5cm}+\mathbb{P}\big\{X_{t_{i_0}}=J_{i_0} \big \vert \mathcal{F}_{t_{i_0},i}\big\}
\times \mathbb{P}\big\{\phi=0 \big\vert X_{t_{i_0}}=J_{i_0},\mathcal{F}_{t_{i_0}}\big\}\\
&\hspace{3cm}\times \mathbb{P}\big\{X_{t_{i_1}}=J_{i_1} \big\vert \phi=0, X_{t_{i_0}}=J_{i_0},\mathcal{F}_{t_{i_0},i} \big\} \\
&\hspace{3.2cm}\vdots \\
&\hspace{3cm} 
\times \mathbb{P}\big\{X_{t_{i_n}}=J_{i_n} \big\vert \phi=0, X_{t_{i_{n-1}}}=J_{i_{n-1}},\dots,X_{t_{i_0}}=J_{i_0},\mathcal{F}_{t_{i_0},i} \big\}.
 \end{align*}
Recall that $\mathbb{P}\big\{X_{t_{i_0}}=J_{i_0}\vert \mathcal{F}_{t_{i_0},i}\big\}=1$ iff $J_{i_0}=i$ and $0$ otherwise. In terms of (\ref{eq:bayesianupdates}),
\begin{align*}
&\hspace{3.5cm}\mathbb{P}\big\{X_{t_{i_0}}=j_{i_0},X_{t_{i_1}}=j_{i_1},..., X_{t_{i_n}}=j_{i_n}  \big\vert \mathcal{F}_{t_{i_0},i} \big\}\\
&\hspace{0cm}=\mathbf{e}_i^{\top}\mathbf{S}(t) e^{\mathbf{G}(t_{i_1}-t_{i_0})} \mathbf{e}_{J_{i_1}} \mathbf{e}_{J_{i_1}}^{\top} e^{\mathbf{G}(t_{i_2}-t_{i_1})} \dots \mathbf{e}_{J_{i_{n-1}}} \mathbf{e}_{J_{i_{n-1}}}^{\top} e^{\mathbf{G}(t_{i_n}-t_{i_{n-1}})} \mathbf{e}_{J_{i_n}}  \mathbf{e}_{J_{i_n}}^{\top}\mathbb{1}\\
&\hspace{0cm}+\mathbf{e}_i^{\top}\big(\mathbf{I}-\mathbf{S}(t) \big) e^{\mathbf{Q}(t_{i_1}-t_{i_0})} \mathbf{e}_{J_{i_1}} \mathbf{e}_{J_{i_1}}^{\top} e^{\mathbf{Q}(t_{i_2}-t_{i_1})} \dots \mathbf{e}_{J_{i_{n-1}}} \mathbf{e}_{J_{i_{n-1}}}^{\top} e^{\mathbf{Q}(t_{i_n}-t_{i_{n-1}})} \mathbf{e}_{J_{i_n}}  \mathbf{e}_{J_{i_n}}^{\top}\mathbb{1}.
\end{align*}
Therefore, starting from equation (\ref{eq:Derivation}) we have following the above that
\begin{align*}
&\mathbb{P}\big\{\tau_{i_1}>t_{i_1},...,\tau_{i_n}>t_{i_n} \big\vert \mathcal{F}_{t_{i_0},i}\big\}\\
&\hspace{1cm}= \mathbf{e}_i^{\top}\mathbf{S}(t) e^{\mathbf{G}(t_{i_1}-t_{i_0})} \Big(\sum_{J_{i_1}\in\Gamma_{i_1}^c} \mathbf{e}_{J_{i_1}} \mathbf{e}_{J_{i_1}}^{\top}\Big) e^{\mathbf{G}(t_{i_2}-t_{i_1})} \\
&\hspace{2cm}\dots \Big(\sum_{J_{i_{n-1}} \in \Gamma_{i_{n-1}}^c} \mathbf{e}_{J_{i_{n-1}}} \mathbf{e}_{J_{i_{n-1}}}^{\top} \Big) e^{\mathbf{G}(t_{i_n}-t_{i_{n-1}})}
\Big(\sum_{J_{i_n}\in \Gamma_{i_n}^c} \mathbf{e}_{J_{i_n}}  \mathbf{e}_{J_{i_n}}^{\top} \Big)\mathbb{1} \\
&\hspace{1cm}+ \mathbf{e}_i^{\top}\big(\mathbf{I}-\mathbf{S}(t) \big) e^{\mathbf{Q}(t_{i_1}-t_{i_0})} \Big(\sum_{J_{i_1}\in\Gamma_{i_1}^c} \mathbf{e}_{J_{i_1}} \mathbf{e}_{J_{i_1}}^{\top}\Big) e^{\mathbf{Q}(t_{i_2}-t_{i_1})} \\
&\hspace{2cm}\dots \Big(\sum_{J_{i_{n-1}} \in \Gamma_{i_{n-1}}^c} \mathbf{e}_{J_{i_{n-1}}} \mathbf{e}_{J_{i_{n-1}}}^{\top} \Big) e^{\mathbf{Q}(t_{i_n}-t_{i_{n-1}})}
\Big(\sum_{J_{i_n}\in \Gamma_{i_n}^c} \mathbf{e}_{J_{i_n}}  \mathbf{e}_{J_{i_n}}^{\top} \Big)\mathbb{1},
\end{align*}
leading to $\overline{F}_{i,t}(t_{i_1},\dots,t_{i_n})$ on account of the fact that $\mathbf{H}_{i_k}=\sum\limits_{J_{i_k}\in\Gamma_{i_k}^c} \mathbf{e}_{J_{i_k}}\mathbf{e}_{J_{i_k}}^{\top}$ and after applying the block partition (\ref{eq:blockpartisi}) for exponential matrices $e^{\mathbf{G}}$ and $e^{\mathbf{Q}}$. \exit
\end{proof}

\medskip

Notice that the distribution forms a non-stationary function of time with the ability to capture heterogeneity and path dependence when conditioning on all previous and current information $\mathcal{F}_{t,i}$. These features are removed when $\mathbf{B}=\mathbf{A}$, in which case, the result reduces to the multivariate phase-type distribution \cite{Assaf1984}.

\begin{prop}
Let $t_{i_n}\geq \dots \geq t_{i_1}\geq t_{i_0}=t\geq 0$ be the time ordering of $(t_1,...,t_n)\in\mathbb{R}_+^n$. The $\mathcal{G}_t-$conditional joint distribution of $\tau_k$ (\ref{eq:MultiPH}) is given by
\begin{equation}\label{eq:mainmph2}
\begin{split}
&\hspace{2.65cm}\overline{F}_{t}(t_{i_1},...,t_{i_n})=\mathbb{P}\big\{\tau_{i_1}>t_{i_1},...,\tau_{i_n}>t_{i_n} \big\vert \mathcal{G}_t\big\}\\
&\hspace{-0.15cm}=\boldsymbol{\pi}^{\top}(t)\Big\{\mathbf{S}(t) \prod_{k=1}^n e^{\mathbf{B}(t_{i_k}-t_{i_{k-1}})}\mathbf{H}_{i_k} + \big[\mathbf{I}-\mathbf{S}(t)\big]  \prod_{k=1}^ne^{\mathbf{A}(t_{i_k}-t_{i_{k-1}})} \mathbf{H}_{i_k}\Big\}\mathbb{1}. 
\end{split}
\end{equation}
\end{prop}
\begin{proof}
It follows from (\ref{eq:relation}) that $\overline{F}_{t}(t_{i_1},...,t_{i_n})=\sum\limits_{i=1}^m \pi_i(t) \overline{F}_{i,t}(t_{i_1},...,t_{i_n})$. \exit
\end{proof}

\begin{cor}
Set $\mathbf{B}=\mathbf{A}$ and $t=0$ in (\ref{eq:mainmph2}). The joint distribution of $\{\tau_k\}$, 
\begin{equation}\label{eq:main1}
\begin{split}
\mathbb{P}\big\{\tau_{i_1}>t_{i_1},...,\tau_{i_n}>t_{i_n})=&\boldsymbol{\pi}^{\top}\Big(\prod_{k=1}^n e^{\mathbf{A}(t_{i_k}-t_{i_{k-1}})}\mathbf{H}_{k}\Big)\mathbb{1},
\end{split}
\end{equation}
coincides with the unconditional multivariate phase-type distribution \cite{Assaf1984}.
\end{cor}

The absolutely continuous component of the distribution $\overline{F}_{i,t}\big(t_{i_1},\dots,t_{i_n}\big)$ (resp. $\overline{F}_{t}\big(t_{i_1},\dots,t_{i_n}\big)$) has a density given by the following theorem.
\begin{theo}
Let $t_{i_n}\geq \dots \geq t_{i_1}\geq t_{i_0}=t\geq 0$ be the time ordering of $(t_1,...,t_n)\in\mathbb{R}_+^n$. The conditional joint density function of $\tau_k$ (\ref{eq:MultiPH}) is given by 
\begin{align*}
f_{i,t}\big(t_{i_1},\dots,t_{i_n}\big)=&(-1)^n\mathbf{e}_i^{\top}\Big\{\mathbf{S}(t)\Big(\prod_{k=1}^{n-1} e^{\mathbf{B}(t_k-t_{k-1})}[\mathbf{B},\mathbf{H}_{i_k}] \Big)e^{\mathbf{B}(t_n-t_{n-1})}\mathbf{B}\mathbf{H}_{i_n} \\
&+\big(\mathbf{I}-\mathbf{S}(t)\big)\Big(\prod_{k=1}^{n-1} e^{\mathbf{A}(t_k-t_{k-1})}[\mathbf{A},\mathbf{H}_{i_k}] \Big)e^{\mathbf{A}(t_n-t_{n-1})}\mathbf{A}\mathbf{H}_{i_n}\Big\}\mathbb{1}, \\[8pt]
f_{t}\big(t_{i_1},\dots,t_{i_n}\big)=&(-1)^n\boldsymbol{\pi}^{\top}(t)\Big\{\mathbf{S}(t)\Big(\prod_{k=1}^{n-1} e^{\mathbf{B}(t_k-t_{k-1})}[\mathbf{B},\mathbf{H}_{i_k}] \Big)e^{\mathbf{B}(t_n-t_{n-1})}\mathbf{B}\mathbf{H}_{i_n} \\
&+\big(\mathbf{I}-\mathbf{S}(t)\big)\Big(\prod_{k=1}^{n-1} e^{\mathbf{A}(t_k-t_{k-1})}[\mathbf{A},\mathbf{H}_{i_k}] \Big)e^{\mathbf{A}(t_n-t_{n-1})}\mathbf{A}\mathbf{H}_{i_n}\Big\}\mathbb{1}. 
\end{align*}
\end{theo}
\begin{proof}
The proof follows from taking partial derivative to $F_{i,t}\big(t_{i_1},\dots,t_{i_n}\big)$:
\begin{align*}
f_{i,t}\big(t_{i_1},\dots,t_{i_n}\big)=&(-1)^n \frac{\partial^n \overline{F}_{i,t}}{\partial t_{i_n}\dots \partial t_{i_1}}\big(t_{i_1},\dots,t_{i_n}\big).
\end{align*}
To justify the claim, we use induction argument. For this purpose, recall that 
\begin{align}\label{eq:perkalian}
\prod_{k=1}^n e^{\mathbf{B}(t_{i_k}-t_{i_{k-1}})}\mathbf{H}_{i_k}=e^{\mathbf{B}(t_{i_1}-t_{i_{0}})}\mathbf{H}_{i_1}e^{\mathbf{B}(t_{i_2}-t_{i_{1}})}\mathbf{H}_{i_2}\prod_{k=3}^n e^{\mathbf{B}(t_{i_k}-t_{i_{k-1}})}\mathbf{H}_{i_k}.
\end{align}

\pagebreak 

\noindent Hence, by (\ref{eq:dervexpm}) and applying integration by part as we did before, we have 
\begin{align*}
\frac{\partial}{\partial t_{i_1}}\prod_{k=1}^n e^{\mathbf{B}(t_{i_k}-t_{i_{k-1}})}\mathbf{H}_{i_k}=&
e^{\mathbf{B}(t_{i_1}-t_{i_{0}})}[\mathbf{B},\mathbf{H}_{i_1}]\prod_{k=2}^n e^{\mathbf{B}(t_{i_k}-t_{i_{k-1}})}\mathbf{H}_{i_k},
\end{align*}
from which the second order partial derivative $\frac{\partial^2}{\partial t_{i_2}\partial t_{i_1}}$ of (\ref{eq:perkalian}) is given by
\begin{align*}
\frac{\partial^2}{\partial t_{i_2}\partial t_{i_1}}\prod_{k=1}^n e^{\mathbf{B}(t_{i_k}-t_{i_{k-1}})}\mathbf{H}_{i_k}=&
e^{\mathbf{B}(t_{i_1}-t_{i_{0}})}[\mathbf{B},\mathbf{H}_{i_1}]\frac{\partial}{\partial t_{i_2}}\Big(\prod_{k=2}^n e^{\mathbf{B}(t_{i_k}-t_{i_{k-1}})}\mathbf{H}_{i_k}\Big)\\
&\hspace{-2.5cm}=e^{\mathbf{B}(t_{i_1}-t_{i_{0}})}[\mathbf{B},\mathbf{H}_{i_1}] e^{\mathbf{B}(t_{i_2}-t_{i_{1}})}[\mathbf{B},\mathbf{H}_{i_2}] \Big(\prod_{k=3}^n e^{\mathbf{B}(t_{i_k}-t_{i_{k-1}})}\mathbf{H}_{i_k}\Big).
\end{align*}
After $(n-1)$steps of taking the partial derivative, one can show that 
\begin{align*}
\frac{\partial^{n-1}}{\partial t_{i_{n-1}}\dots\partial t_{i_1}}\prod_{k=1}^n e^{\mathbf{B}(t_{i_k}-t_{i_{k-1}})}\mathbf{H}_{i_k} =\Big(\prod_{k=1}^{n-1} e^{\mathbf{B}(t_{i_k}-t_{i_{k-1}})}[\mathbf{B},\mathbf{H}_{i_k}]\Big) e^{\mathbf{B}(t_{i_n}-t_{i_{n-1}})} \mathbf{H}_{i_n}.
\end{align*}
The claim is established on account of (\ref{eq:dervexpm}) and the fact that 
\begin{align*}
\frac{\partial^n \overline{F}_{i,t}}{\partial t_{i_n}\dots \partial t_{i_1}} (t_{i_1},\dots,t_{i_n})=&
\mathbf{e}_i^{\top}\Big\{\mathbf{S}(t)\frac{\partial^n}{\partial t_{i_n}\dots \partial t_{i_1}} \Big(\prod_{k=1}^n e^{\mathbf{B}(t_{i_k}-t_{i_{k-1}})}\mathbf{H}_{i_k} \Big) \\
&\hspace{-1cm}+ \big[\mathbf{I}-\mathbf{S}(t)\big]  
\frac{\partial^n}{\partial t_{i_n}\dots \partial t_{i_1}} \Big(\prod_{k=1}^ne^{\mathbf{A}(t_{i_k}-t_{i_{k-1}})} \mathbf{H}_{i_k}\Big)\Big\}\mathbb{1}.
\end{align*}
The proof for the $\mathcal{G}_t-$conditional joint density $f_t(t_{i_1},\dots,t_{i_n})$ follows from 
\begin{align*}
f_t(t_{i_1},\dots,t_{i_n})=\sum_{i=1}^m \pi_i(t) \frac{\partial^n}{\partial t_{i_n}\dots \partial t_{i_1}} \overline{F}_{i,t}(t_{i_1},\dots,t_{i_n}). \exit
\end{align*}

\noindent However, due to complexity of the joint distributions, the singular component of $\overline{F}_{i,t}(t_{i_1},\dots,t_{i_n})$ (resp. $\overline{F}_{t}(t_{i_1},\dots,t_{i_n})$) is more complicated to get in closed form. 
\end{proof}

\medskip

Following (\ref{eq:maincdf1}) and (\ref{eq:mainmph2}), we see that the distributions are presented in terms of a generalized mixture of the multivariate phase-type distributions \cite{Assaf1984}. 

They are uniquely characterized by the probability $\boldsymbol{\pi}$ of starting the mixture process $X$ (\ref{eq:mixture}) on the state space $\mathbb{S}$, the speed of the process, which is represented by the phase generator matrices $\mathbf{B}$ and $\mathbf{A}$, and by the switching probability matrix $\mathbf{S}$. They coincide with \cite{Assaf1984} when the process never repeatedly changes the speed, i.e., $\mathbf{B}=\mathbf{A}$ and when sending $t$ to zero. As in the univariate case, the multivariate distributions have closure and dense properties, which can be established in similar ways to the univariate analogs using matrix analytic approach \cite{Assaf1984}. We refer among others to \cite{Neuts1981}, \cite{Assaf1982}, \cite{He} and \cite{Rolski} for the Markov model, and to \cite{Surya2018} for the mixture model. As a result, we have the following.

\begin{theo}[Closure and dense properties]
The multivariate probability distribution (\ref{eq:mainmph2}) forms a dense class of distributions on $\mathbb{R}_+^n$, which is closed under finite convex mixtures and finite convolutions.
\end{theo}

\section{Some explicit examples}
This section discusses some explicit examples of the main results presented in Section \ref{sec:mainsection}, particularly on the bivariate distributions. Using the closed form density functions (\ref{eq:jointpdfit}) and (\ref{eq:jointpdft}), we discuss the mixtures of exponential distributions, Marshall-Olkin exponential distributions, and their generalization. 

\begin{Ex}[Mixture of exponential distribution]

Consider the mixture process $X$ (\ref{eq:mixture}) defined on the state space $\mathbb{S}=\{1,2,3\}\cup \{\Delta\}$ with stochastically closed sets $\boldsymbol{\Gamma}_1=\{2,\Delta\}$ and $\boldsymbol{\Gamma}_2=\{3,\Delta\}$. Assume that the speed of the mixture process is represented by the following phase generator matrices:
\begin{equation*}
\mathbf{B} = \left(\begin{array}{ccc}
  -(b_1+b_2) & b_1 & b_2 \\
  0 & -b_2 & 0 \\
  0 & 0 & -b_1
\end{array}\right)
\quad \textrm{and} \quad
\mathbf{A} = \left(\begin{array}{ccc}
  -(a_1+a_2) & a_1 & a_2 \\
  0 & -a_2 & 0 \\
  0 & 0 & -a_1
\end{array}\right).
\end{equation*}
It is straightforward to derive from the state space representation that
\begin{align*}
\mathbf{H}_1= \left(\begin{array}{ccc}
 1 & 0 & 0 \\
  0 & 0 & 0 \\
   0 & 0  & 1
\end{array}\right)
\quad \textrm{and} \quad 
\mathbf{H}_2= \left(\begin{array}{ccc}
 1 & 0 & 0 \\
  0 & 1 & 0 \\
   0 & 0  & 0
\end{array}\right).
\end{align*}
After some calculations, the matrices $[\mathbf{A},\mathbf{H}_k]$ and $\mathbf{A}\mathbf{H}_k$, $k=1,2$, are given by
\begin{align*}
[\mathbf{A},\mathbf{H}_1]=
\left(\begin{array}{ccc}
 0 & -a_1 & 0 \\
  0 & 0 & 0 \\
   0 & 0  & 0
\end{array}\right)
\quad \textrm{and}& \quad 
[\mathbf{A},\mathbf{H}_2]=
\left(\begin{array}{ccc}
 0 & 0 & -a_2 \\
  0 & 0 & 0 \\
   0 & 0  & 0
\end{array}\right) \\[8pt]
\mathbf{A}\mathbf{H}_1=
\left(\begin{array}{ccc}
 -(a_1+a_2) & 0 & a_2 \\
  0 & 0 & 0 \\
   0 & 0  & -a_1
\end{array}\right)
\quad \textrm{and}& \quad
\mathbf{A}\mathbf{H}_2=
\left(\begin{array}{ccc}
 -(a_1+a_2) & a_1 & 0 \\
  0 & -a_2 & 0 \\
   0 & 0  & 0
\end{array}\right).
\end{align*}
Similarly defined for $[\mathbf{B},\mathbf{H}_k]$ and $\mathbf{B}\mathbf{H}_k$, for $k=1,2$.
Set the matrix $\mathbf{S}=\textrm{diag}(p_1,p_2,p_3)$, with $0<p_k<1$, for $k=1,2,3$, whilst the initial probability $\boldsymbol{\pi}$ has mass one on the state $1$, i.e., $\boldsymbol{\pi}=\mathbf{e}_1$. It is straightforward to check that the condition (\ref{eq:singularcond}) is clearly satisfied implying that the joint density function (\ref{eq:jointpdfit}) has zero singular component. Hence, following (\ref{eq:jointpdfit}) we have
\begin{align*}
f_{\tau_1,\tau_2}(t_1,t_2)=p_1 b_1 e^{-b_1 t_1}b_2 e^{-b_2 t_2} + (1-p_1)a_1 e^{-a_1 t_1}a_2 e^{-a_2 t_2},
\end{align*}
for $t_1,t_2\geq 0$. The marginal distribution of $\tau_1$and $\tau$ are given respectively by
\begin{align*}
f_{\tau_1}(t_1)=&p_1 b_1 e^{-b_1 t_1} + (1-p_1)a_1 e^{-a_1 t_1}\\
f_{\tau_2}(t_2)=&p_1 b_2 e^{-b_2 t_2} + (1-p_1) a_2 e^{-a_2 t_2}.
\end{align*}

Hence, clearly, as $f_{\tau_1,\tau_2}(t_1,t_2)\neq f_{\tau_1}(t_1)f_{\tau_2}(t_2)$, it follows that the exit times $\tau_1$ and $\tau_2$ are not independent under the mixture model. They are independent if and only if $a_1=b_1=b_2=a_2$, in which case the mixture corresponds to a simple Markov jump process. See the example on p. 691 in \cite{Assaf1984} and p. 59 in \cite{He}. 

Furthermore, when conditioning on the information set $\mathcal{F}_{t,i}$ with $i=1$, the conditional joint density function $f_{1,t}(t_1,t_2)$ is given for $t_1,t_2\geq t\geq 0$  by  
\begin{equation}\label{eq:ex1}
\begin{split}
f_{1,t}(t_1,t_2)=&s_1(t) e^{(b_1+b_2)t} b_1 e^{-b_1 t_1} b_2 e^{-b_2t_2} \\
&+ (1-s_1(t))  e^{(a_1+a_2)t} a_1 e^{-a_1 t_1} a_2 e^{-a_2t_2},
\end{split}
\end{equation}
where the switching probability $s_1(t)$ is defined for all $t\geq 0$ by
\begin{align*}
s_1(t)=\frac{p_1e^{-(b_1+b_2)t}}{p_1e^{-(b_1+b_2)t} + (1-p_1)e^{-(a_1+a_2)t}}.
\end{align*}
Observe that, on the event $\{\textrm{min}\{\tau_1,\tau_2\}>t\}$, one can check that $s_1(t)\rightarrow 0$ (resp. $1$) as $t\rightarrow \infty$ if $b_1+b_2 > \; (\textrm{resp. $<$})\; a_1+a_2$, implying that the mixture process moves at the slower speed $\mathbf{A}$ (resp. $\mathbf{B}$) in the long run as a Markov process.  

Given that $\Gamma_1^c\cap\Gamma_2^c=\{1\}$, we have $\pi_1(t)=1$ for all $t\geq 0$. Hence, the density function $f_t(t_1,t_2)$ (\ref{eq:jointpdft}) has therefore the same expression as (\ref{eq:ex1}).
\end{Ex}

\begin{Ex}[Mixture of Marshall-Olkin distribution]
Consider a mixture process $X$ (\ref{eq:mixture}) with the same state space $\mathbb{S}$ and stochastically closed sets $\Gamma_1$ and $\Gamma_2$ as defined above. Let the speed of the mixture process be given by
\begin{align*}
\mathbf{B} =& \left(\begin{array}{ccc}
  -(b_1+b_2+b_3) & b_1 & b_2 \\
  0 & -(b_2+b_3) & 0 \\
  0 & 0 & -(b_1+b_3)
\end{array}\right)\\[8pt]
\mathbf{A} =& \left(\begin{array}{ccc}
  -(a_1+a_2+a_3) & a_1 & a_2 \\
  0 & -(a_2+a_3) & 0 \\
  0 & 0 & -(a_1+a_3)
\end{array}\right).
\end{align*}
Set the matrix $\mathbf{S}=\textrm{diag}(p_1,p_2,p_3)$, with $0<p_k<1$, for $k=1,2,3$, and whilst the initial distribution has mass one on the state $1$, i.e., $\pi_1=1$. Following (\ref{eq:singularcond}), the joint density $f_{1,t}(t_1,t_2)$ has singular part on the set $\{(t_1,t_2):t_2=t_1\}$. By the same approach as above, we have following Theorem \ref{theo:maintheo} and Corollary \ref{cor:jointCDFit}:
\begin{align*}
f_{1,t}^{(1)}(t_1,t_2)=&s_1(t)b_2(b_1+b_3)e^{-b_1(t_1-t)} e^{-b_2(t_2-t)} e^{-b_3(t_1-t_2)}
\\&+ \big(1-s_1(t)\big)a_2(a_1+a_3)e^{-a_1(t_1-t)} e^{-a_2(t_2-t)} e^{-a_3(t_1-t_2)}\\[8pt]
f_{1,t}^{(2)}(t_1,t_2)=&s_1(t)b_1(b_2+b_3)e^{-b_1(t_1-t)}e^{-b_2(t_2-t)} e^{-b_3(t_2-t_1)}\\
&+ \big(1-s_1(t)\big)a_1(a_2+a_3)e^{-a_1(t_1-t)}e^{-a_2(t_2-t)} e^{-a_3(t_2-t_1)}\\[8pt]
f_{1,t}^{(0)}(t_1,t_1)=&s_1(t) b_3 e^{-(b_1+b_2+b_3)(t_1-t)} 
+ \big(1-s_1(t) \big) a_3 e^{-(a_1+a_2+a_3)(t_1-t)},
\end{align*}
whereas the switching probability $s_1(t)$ is given for all $t\geq 0$ by
\begin{align*}
s_1(t)=\frac{p_1e^{-(b_1+b_2+b_3)t}}{p_1e^{-(b_1+b_2+b_3)t} + (1-p_1)e^{-(a_1+a_2+a_3)t}}.
\end{align*}
\end{Ex}

In order to take advantage to the structure of the generator matrices, let
\begin{equation}\label{eq:intensityex}
\mathbf{B} = \left(\begin{array}{ccc}
 \mathbf{B}_{11} & \mathbf{B}_{12} & \mathbf{B}_{13} \\
  \mathbf{0} & \mathbf{B}_{22} & \mathbf{0} \\
   \mathbf{0} & \mathbf{0}  & \mathbf{B}_{33}\
\end{array}\right) 
\quad \textrm{and} \quad 
\mathbf{A} = \left(\begin{array}{ccc}
 \mathbf{A}_{11} & \mathbf{A}_{12} & \mathbf{A}_{13} \\
  \mathbf{0} & \mathbf{A}_{22} & \mathbf{0} \\
   \mathbf{0} & \mathbf{0}  & \mathbf{A}_{33}
\end{array}\right).
\end{equation}
The generator matrices $\mathbf{A}$ and $\mathbf{B}$ are nonsingular if and only if $\mathbf{A}_{11}$, $\mathbf{A}_{22}$, $\mathbf{A}_{33}$, $\mathbf{B}_{11}$, $\mathbf{B}_{22}$ and $\mathbf{B}_{33}$ are all nonsingular. The matrices $\mathbf{H}_1$ and $\mathbf{H}_2$ are given by
\begin{align*}
\mathbf{H}_1= \left(\begin{array}{ccc}
 \mathbf{I} & \mathbf{0} & \mathbf{0} \\
  \mathbf{0} & \mathbf{0} & \mathbf{0} \\
   \mathbf{0} & \mathbf{0}  & \mathbf{I}
\end{array}\right)
\quad \textrm{and} \quad 
\mathbf{H}_2= \left(\begin{array}{ccc}
 \mathbf{I} & \mathbf{0} & \mathbf{0} \\
  \mathbf{0} & \mathbf{I} & \mathbf{0} \\
   \mathbf{0} & \mathbf{0}  & \mathbf{0}
\end{array}\right).
\end{align*}
After some calculations the matrix $[\mathbf{A},\mathbf{H}_k]$ and $\mathbf{A}\mathbf{H}_k$, $k=1,2$, are given by
\begin{align*}
[\mathbf{A},\mathbf{H}_1]=&
\left(\begin{array}{ccc}
 \mathbf{0} & -\mathbf{A}_{12} & \mathbf{0} \\
  \mathbf{0} & \mathbf{0} & \mathbf{0} \\
   \mathbf{0} & \mathbf{0}  & \mathbf{0}
\end{array}\right)
\quad \textrm{and} \quad 
[\mathbf{A},\mathbf{H}_2]=
\left(\begin{array}{ccc}
 \mathbf{0} & \mathbf{0} & -\mathbf{A}_{13} \\
  \mathbf{0} & \mathbf{0} & \mathbf{0} \\
   \mathbf{0} & \mathbf{0}  & \mathbf{0}
\end{array}\right) \\[8pt]
\mathbf{A}\mathbf{H}_1=&
\left(\begin{array}{ccc}
 \mathbf{A}_{11} & \mathbf{0} & \mathbf{A}_{13} \\
  \mathbf{0} & \mathbf{0} & \mathbf{0} \\
   \mathbf{0} & \mathbf{0}  & \mathbf{A}_{33}
\end{array}\right)
\quad \textrm{and} \quad
\mathbf{A}\mathbf{H}_2=
\left(\begin{array}{ccc}
 \mathbf{A}_{11} & \mathbf{A}_{12} & \mathbf{0} \\
  \mathbf{0} & \mathbf{A}_{22} & \mathbf{0} \\
   \mathbf{0} & \mathbf{0}  & \mathbf{0}
\end{array}\right).
\end{align*}
Similarly defined for $[\mathbf{B},\mathbf{H}_k]$ and $\mathbf{B}\mathbf{H}_k$, for $k=1,2$. A rather long calculations using infinite series representation of exponential matrix shows following (\ref{eq:jointpdfit}),
\begin{eqnarray} \label{eq:jointpdfitex}
f_{i,t}(t_1,t_2)=
\begin{cases}
f_{i,t}^{(1)}(t_1,t_2), &\; \textrm{if $t_1\geq t_2\geq t \geq 0$} \\[8pt]
f_{i,t}^{(2)}(t_1,t_2), &\; \textrm{if $t_2\geq t_1\geq t \geq 0$}\\[8pt]
f_{i,t}^{(0)}(t_1,t_1), &\; \textrm{if $t_1= t_2\geq t \geq 0$}, \\
\end{cases}
\quad \textrm{for $i\in \Gamma_1^c \cap \Gamma_2^c$}.
\end{eqnarray}
where the absolutely continuous parts $f_{i,t}^{(1)}(t_1,t_2)$ and $f_{i,t}^{(2)}(t_1,t_2)$ are given by
 \begin{align*}
 f_{i,t}^{(1)}(t_1,t_2)=&-\mathbf{e}_i^{\top}\Big\{\mathbf{S}_{11}(t)e^{\mathbf{B}_{11}(t_2-t)}\mathbf{B}_{13}e^{\mathbf{B}_{33}(t_1-t_2)}\mathbf{B}_{33} \\
 &\hspace{2cm}+\big[\mathbf{I}-\mathbf{S}_{11}(t)\big]e^{\mathbf{A}_{11}(t_2-t)}\mathbf{A}_{13}e^{\mathbf{A}_{33}(t_1-t_2)}\mathbf{A}_{33}\Big\}\mathbb{1},\\[8pt]
f_{i,t}^{(2)}(t_1,t_2)=&-\mathbf{e}_i^{\top}\Big\{\mathbf{S}_{11}(t)e^{\mathbf{B}_{11}(t_1-t)}\mathbf{B}_{12}e^{\mathbf{B}_{22}(t_2-t_1)}\mathbf{B}_{22} \\
 &\hspace{2cm}+\big[\mathbf{I}-\mathbf{S}_{11}(t)\big]e^{\mathbf{A}_{11}(t_1-t)}\mathbf{A}_{12}e^{\mathbf{A}_{22}(t_2-t_1)}\mathbf{A}_{22}\Big\}\mathbb{1}, 
 \end{align*}
 whereas the singular component $f_{i,t}^{(0)}(t_1,t_1)$ is defined by the function
 \begin{align*}
  f_{i,t}^{(0)}(t_1,t_1)=&-\mathbf{e}_i^{\top}\Big\{\mathbf{S}_{11}(t) e^{\mathbf{B}_{11}(t_1-t)}\big(\mathbf{B}_{11} + \mathbf{B}_{12} + \mathbf{B}_{13}\big)\\
 &\hspace{2cm}+\big[\mathbf{I}-\mathbf{S}_{11}(t)\big]e^{\mathbf{A}_{11}(t_1-t)}\big(\mathbf{A}_{11} + \mathbf{A}_{12} + \mathbf{A}_{13}\big)\Big\}\mathbb{1}. 
 \end{align*}
Note that $\mathbf{S}_{11}(t)$ denotes the switching probability matrix of $X$ on $\Gamma_1^c\cap \Gamma_2^c$.
 
 \pagebreak
 
 Thus, the distribution $f_{i,t}(t_1,t_2)$ has no singular component $f_{i,t}^{(0)}(t_1,t_2)$ iff
 \begin{align*}
 \mathbf{A}_{11}+\mathbf{A}_{12}+\mathbf{A}_{13}=\mathbf{0}=\mathbf{B}_{11}+\mathbf{B}_{12}+\mathbf{B}_{13}.
 \end{align*}

Denote by $\boldsymbol{\alpha}$ the restriction of the probability distribution $\boldsymbol{\pi}$ on the set $\Gamma_1^c\cap \Gamma_2^c$ such that $\boldsymbol{\pi}=\big(\boldsymbol{\alpha},\mathbf{0}\big)$. The Bayesian updates $\boldsymbol{\pi}(t)$ on $\Gamma_1^c\cap \Gamma_2^c$ is defined by $\boldsymbol{\alpha}(t)$.
 
The conditional joint density $f_{t}(t_1,t_2)$ of the exit times $\tau_1$ and $\tau_2$ is given by 
\begin{eqnarray} \label{eq:jointpdfitex2}
f_{t}(t_1,t_2)=
\begin{cases}
f_{t}^{(1)}(t_1,t_2), &\; \textrm{if $t_1\geq t_2\geq t \geq 0$} \\[8pt]
f_{t}^{(2)}(t_1,t_2), &\; \textrm{if $t_2\geq t_1\geq t \geq 0$}\\[8pt]
f_{t}^{(0)}(t_1,t_1), &\; \textrm{if $t_1= t_2\geq t \geq 0$}, \\
\end{cases}
\end{eqnarray}
where the subdensity functions $f_{t}^{(1)}(t_1,t_2)$, $f_{t}^{(2)}(t_1,t_2)$ and $f_{t}^{(0)}(t_1,t_2)$ are
 \begin{align*}
 f_{t}^{(1)}(t_1,t_2)=&-\boldsymbol{\alpha}^{\top}(t)\Big\{\mathbf{S}_{11}(t)e^{\mathbf{B}_{11}(t_2-t)}\mathbf{B}_{13}e^{\mathbf{B}_{33}(t_1-t_2)}\mathbf{B}_{33} \\
 &\hspace{2cm}+\big[\mathbf{I}-\mathbf{S}_{11}(t)\big]e^{\mathbf{A}_{11}(t_2-t)}\mathbf{A}_{13}e^{\mathbf{A}_{33}(t_1-t_2)}\mathbf{A}_{33}\Big\}\mathbb{1},\\[8pt]
f_{t}^{(2)}(t_1,t_2)=&-\boldsymbol{\alpha}^{\top}(t)\Big\{\mathbf{S}_{11}(t)e^{\mathbf{B}_{11}(t_1-t)}\mathbf{B}_{12}e^{\mathbf{B}_{22}(t_2-t_1)}\mathbf{B}_{22} \\
 &\hspace{2cm}+\big[\mathbf{I}-\mathbf{S}_{11}(t)\big]e^{\mathbf{A}_{11}(t_1-t)}\mathbf{A}_{12}e^{\mathbf{A}_{22}(t_2-t_1)}\mathbf{A}_{22}\Big\}\mathbb{1}, \\[8pt]
 f_{t}^{(0)}(t_1,t_1)=&-\boldsymbol{\alpha}^{\top}(t)\Big\{\mathbf{S}_{11}(t) e^{\mathbf{B}_{11}(t_1-t)}\big(\mathbf{B}_{11} + \mathbf{B}_{12} + \mathbf{B}_{13}\big)\\
 &\hspace{2cm}+\big[\mathbf{I}-\mathbf{S}_{11}(t)\big]e^{\mathbf{A}_{11}(t_1-t)}\big(\mathbf{A}_{11} + \mathbf{A}_{12} + \mathbf{A}_{13}\big)\Big\}\mathbb{1}.
 \end{align*}

The marginal distribution of $\tau_k$, $k=1,2$, can be established in the same approach for the univariate case. It is given for $s\geq t\geq 0$ and $i\in \Gamma_1^c\cap \Gamma_2^c$ by 
 \begin{align}
 \mathbb{P}\{\tau_k>s \big\vert \mathcal{F}_{t,i}\}=&\mathbf{e}_i^{\top} \Big(\mathbf{S}_{11}(t) e^{\mathbf{B}^{(k)}(s-t)} + \big[\mathbf{I}-\mathbf{S}_{11}(t)\big] e^{\mathbf{A}^{(k)}(s-t)}  \Big)\mathbb{1},\\[8pt]
  \mathbb{P}\{\tau_k>s \big\vert \mathcal{G}_{t}\}=&\boldsymbol{\pi}^{\top}(t) \Big(\mathbf{S}_{11}(t) e^{\mathbf{B}^{(k)}(s-t)} + \big[\mathbf{I}-\mathbf{S}_{11}(t)\big] e^{\mathbf{A}^{(k)}(s-t)}  \Big)\mathbb{1},
  \end{align}
  where the phase-generator matrices $\mathbf{B}^{(k)}$ and $\mathbf{A}^{(k)}$, for $k=1,2$, are defined by
  \begin{align*}
\mathbf{B}^{(1)}=  \left(\begin{array}{cc}
 \mathbf{B}_{11} & \mathbf{B}_{13} \\
  \mathbf{0} &  \mathbf{B}_{33} 
\end{array}\right)
\quad \textrm{and}& \quad 
\mathbf{A}^{(1)}=  \left(\begin{array}{cc}
 \mathbf{A}_{11} & \mathbf{A}_{13} \\
  \mathbf{0} &  \mathbf{A}_{33} 
\end{array}\right), \\[8pt]
\mathbf{B}^{(2)}=  \left(\begin{array}{cc}
 \mathbf{B}_{11} & \mathbf{B}_{12} \\
  \mathbf{0} &  \mathbf{B}_{22} 
\end{array}\right)
\quad \textrm{and}& \quad 
\mathbf{A}^{(2)}=  \left(\begin{array}{cc}
 \mathbf{A}_{11} & \mathbf{A}_{12} \\
  \mathbf{0} &  \mathbf{A}_{22} 
\end{array}\right).
\end{align*}

\section{Conclusions}

We have introduced a generalization of the multivariate phase-type distributions \cite{Assaf1984} under the mixture of absorbing Markov jump processes moving at different speeds on the same finite state space. Such mixture process was proposed in \cite{Frydman2008}, and was discussed in further details in \cite{Surya2018}. The new distributions form non-stationary function of time and have the ability to capture heterogeneity and the past information of the process, when conditioning on its past observation. The attribution of path dependence is due to the non-Markov property of the process. Identities are presented explicit in terms of the Bayesian updates of switching probability and the probability of starting the process in any of the transient phases at any given time, the likelihoods, and the intensity matrices of the underlying processes, despite the fact that the mixture itself is non-Markov. When the underlying Markov processes move at the same speed, in which case the mixture becomes a simple Markov jump process, heterogeneity and path dependence are removed from the identities, and the distributions reduce to \cite{Assaf1984}.

The results presented in this paper can be extended in a natural way to the mixture of a finite number of absorbing Markov jump processes moving at different speeds. Given their availability in explicit form and fine properties, the new distributions should be able to offer appealing features for applications.

\end{document}